\documentclass[11pt,a4paper,final, dvipsnames]{amsart}

\input xy
\xyoption{all}

\DeclareMathAlphabet{\mathpzc}{OT1}{pzc}{m}{it}

\usepackage{amsfonts,amsmath,amssymb,indentfirst,mathrsfs,amscd}
\usepackage[mathscr]{eucal}
\usepackage{graphicx}
\usepackage{nicefrac}
\usepackage{amsthm}

\usepackage{color}
\usepackage[utf8]{inputenc}
\usepackage[spanish, english]{babel}
\usepackage{booktabs}
\usepackage{multirow}
\usepackage{verbatim}

\usepackage{calrsfs}
\DeclareMathAlphabet{ \mathcal}{OMS}{zplm}{m}{n}

\usepackage[all]{xy}
\usepackage{graphicx}
\usepackage{stmaryrd}
\usepackage{enumitem}

\usepackage{empheq}
\usepackage{color}
\usepackage{xcolor}
\usepackage{dsfont}

\usepackage[
            colorlinks=true,
            urlcolor=Blue,
            breaklinks=true,
            pdftitle={},
            pdfauthor={}]{hyperref}
\usepackage{color}
\definecolor{tocolor}{rgb}{.1,.1,.5}
\definecolor{urlcolor}{rgb}{.2,.2,.6}
\definecolor{linkcolor}{rgb}{.0,.3,.6}
\definecolor{citecolor}{rgb}{.6,.2,.2}
\hypersetup{colorlinks=true, urlcolor=urlcolor, linkcolor=linkcolor, citecolor=citecolor}

\author[]{\'Angel Gonz\'alez-Prieto and Marina Logares}
\address{ETSI de Sistemas Inform\'aticos, Universidad Polit\'ecnica de Madrid, Calle Alan Turing s/n (Carretera de Valencia Km 7), 28031 Madrid, Spain}\email{angel.gonzalez.prieto@upm.es}
\address{Facultad de Ciencias Matem\'aticas, Universidad Complutense  de  Madrid, Plaza Ciencias  3, 28040 Madrid Spain.}\email{mlogares@ucm.es}

\title[]{On character varieties of singular manifolds}


\DeclareMathOperator{\coker}{coker\,}
\DeclareMathOperator{\Hom}{Hom\,}           
\DeclareMathOperator{\tr}{Tr\,}             

\DeclareMathOperator{\Gr}{Gr}
\DeclareMathOperator{\Sym}{Sym}

\usepackage{amsmath}
\begin{document}

\newtheorem{thm}{Theorem}[section]
\newtheorem{prop}[thm]{Proposition}
\newtheorem{lem}[thm]{Lemma}
\newtheorem{cor}[thm]{Corollary}
\newtheorem{conjecture}{Conjecture}
\newtheorem*{theorem*}{Theorem}

\theoremstyle{definition}
\newtheorem{defn}[thm]{Definition}
\newtheorem{ex}[thm]{Example}
\newtheorem{as}{Assumption}

\theoremstyle{remark}
\newtheorem{rmk}[thm]{Remark}

\theoremstyle{remark}
\newtheorem*{prf}{Proof}

\newcommand{\iacute}{\'{\i}} 
\newcommand{\norm}[1]{\lVert#1\rVert} 

\newcommand{\lto}{\longrightarrow}
\newcommand{\hra}{\hookrightarrow}

\newcommand{\suchthat}{\;\;|\;\;}
\newcommand{\dbar}{\overline{\partial}}

\newcommand{\cA}{\mathcal{A}}
\newcommand{\cC}{\mathcal{C}}
\newcommand{\cD}{\mathcal{D}}
\newcommand{\cE}{\mathcal{E}}
\newcommand{\cF}{\mathcal{F}}
\newcommand{\cG}{\mathcal{G}} 
\newcommand{\cI}{\mathcal{I}} 
\newcommand{\cO}{\mathcal{O}} 
\newcommand{\cM}{\mathcal{M}} 
\newcommand{\cN}{\mathcal{N}} 
\newcommand{\cP}{\mathcal{P}} 
\newcommand{\cQ}{\mathcal{Q}} 
\newcommand{\cS}{\mathcal{S}} 
\newcommand{\cU}{\mathcal{U}} 
\newcommand{\cJ}{\mathcal{J}}
\newcommand{\cX}{\mathcal{X}}
\newcommand{\cT}{\mathcal{T}}
\newcommand{\cV}{\mathcal{V}}
\newcommand{\cW}{\mathcal{W}}
\newcommand{\cB}{\mathcal{B}}
\newcommand{\cR}{\mathcal{R}}
\newcommand{\cH}{\mathcal{H}}
\newcommand{\cZ}{\mathcal{Z}}
\newcommand{\D}{\bar{B}}
\newcommand{\Ss}{\mathcal{D}}

\newcommand{\Ker}{\textrm{Ker}\,}
\renewcommand{\coker}{\textrm{coker}\,}

\newcommand{\ext}{\mathrm{ext}} 
\newcommand{\x}{\times}

\newcommand{\mM}{\mathscr{M}} 

\newcommand{\CC}{\mathbb{C}} 
\newcommand{\QQ}{\mathbb{Q}} 
\newcommand{\FF}{\mathbb{F}} 
\newcommand{\PP}{\mathbb{P}} 
\newcommand{\HH}{\mathbb{H}} 
\newcommand{\RR}{\mathbb{R}} 
\newcommand{\ZZ}{\mathbb{Z}} 
\newcommand{\NN}{\mathbb{N}} 
\newcommand{\DD}{\mathbb{D}} 

\renewcommand{\lg}{\mathfrak{g}} 
\newcommand{\lh}{\mathfrak{h}} 
\newcommand{\lu}{\mathfrak{u}} 
\newcommand{\la}{\mathfrak{a}} 
\newcommand{\lb}{\mathfrak{b}} 
\newcommand{\lm}{\mathfrak{m}} 
\newcommand{\lgl}{\mathfrak{gl}} 
\newcommand{\lZ}{\mathfrak{Z}} 
\newcommand\Unit[1]{\mathds{1}_{#1}}

\newcommand{\too}{\longrightarrow}
\newcommand{\imat}{\sqrt{-1}} 
\newcommand{\tinyclk}{{\scriptscriptstyle \Taschenuhr}} 
\newcommand\restr[2]{\left.#1\right|_{#2}}
\newcommand\rtorus[1]{{\mathbb{T}^{#1}}}
\newcommand\actPartial{\overline{\partial}}
\newcommand\handle[2]{\mathcal{A}^{#1}_{#2}}

\newcommand\pastingArea[2]{S^{#1-1} \times \bar{B}^{#2-#1}}
\newcommand\pastingAreaPlus[2]{S^{#1} \times \bar{B}^{#2-#1-1}}
\newcommand\coordvector[2]{\left.\frac{\partial}{\partial {#1}}\right|_{#2}}

\newcommand\Sets{\textbf{Set}}
\newcommand\Cat{\textbf{Cat}}
\newcommand\Top{\textbf{Top}}
\newcommand\Topc{\textbf{Top}_c}
\newcommand\Toppc{\textbf{Topp}_c}
\newcommand\HTop{\textbf{HoTop}}
\newcommand\SkTop{\textbf{SkTop}}
\newcommand\TopHLC{\textbf{Top}_{hlc}}
\newcommand\TopS{\textbf{Top}_\star}
\newcommand\Diff{\textbf{Diff}}
\newcommand\Diffc{\textbf{Diff}_c}

\newcommand\BaseBord[3]{\mathbf{Bd}_{{#1 #3}}^{#2}}
\newcommand\Bord[1]{\BaseBord{#1}{}{}}
\newcommand\NBord[2]{\mathbf{NBdp}_{{#1}}(#2)}
\newcommand\NBordnp[1]{\mathbf{NBdp}_{{#1}}}
\newcommand\BordC[1]{\BaseBord{#1}{clr}{}}
\newcommand\BordE[1]{\BaseBord{}{}{}(#1)}
\newcommand\EBord[2]{\BaseBord{#1}{#2}{}}
\newcommand\Bordo[1]{\BaseBord{#1}{or}{}}
\newcommand\Bordp[1]{\mathbf{Bdp}_{{#1}}}
\newcommand\Bordpar[2]{\mathbf{Bd}_{{#1}}(#2)}
\newcommand\Bordppar[2]{\mathbf{Bdp}_{{#1}}(#2)}

\newcommand\Tubo[1]{\mathbf{Tb}_{#1}^0}
\newcommand\Tub[1]{\mathbf{Tb}_{#1}}
\newcommand\ETub[2]{\mathbf{Tb}_{#1}^{#2}}
\newcommand\Tubp[1]{\mathbf{Tbp}_{#1}}
\newcommand\Tubpo[1]{\mathbf{Tbp}_{#1}^0}
\newcommand\Tubppar[2]{\mathbf{Tbp}_{#1}(#2)}

\newcommand\Embc{\textbf{Emb}_c}
\newcommand\EEmbc[1]{\textbf{Emb}_c^{#1}}
\newcommand\Embpc{\textbf{Embp}_c}
\newcommand\Embparc[1]{\textbf{Emb}_c(#1)}
\newcommand\Embpparc[1]{\textbf{Embp}_c(#1)}
\newcommand\OpEmb{\textbf{OEmb}}
\newcommand\OpEmbDiff{\textbf{OEmbDiff}}
\newcommand\OpEEmb[1]{\textbf{OEmb}^{#1}}
\newcommand\OpEEmbc[1]{\textbf{OEmb}^{#1}_c}
\newcommand\OpEmbc{\textbf{OEmb}_c}
\newcommand\OpEmbpc{\textbf{OEmbp}_c}
\newcommand\OpEmbp{\textbf{OEmbp}}
\newcommand\SkOpEmb{\textbf{SkOEmb}}

\newcommand\CDefc{\textbf{CDef}_c}
\newcommand\CEDefc[1]{\textbf{CDef}_c^{#1}}
\newcommand\CDefpc{\textbf{CDefp}_c}
\newcommand\CDefparc[1]{\textbf{CDef}_c(#1)}
\newcommand\CDefpparc[1]{\textbf{CDefp}_c(#1)}

\newcommand\CEmbc{\textbf{CEmb}_c}
\newcommand\Conc{\textbf{Con}_c}
\newcommand\Conpc{\textbf{Conp}_c}
\newcommand\CEmb{\textbf{CEmb}}
\newcommand\CEEmbc[1]{\textbf{CEmb}_c^{#1}}
\newcommand\CEmbpc{\textbf{CEmbp}_c}
\newcommand\CEmbparc[1]{\textbf{CEmb}_c(#1)}
\newcommand\CEmbpparc[1]{\textbf{CEmbp}_c(#1)}

\newcommand\cSp{\cS_p}
\newcommand\cSpar[1]{\cS_{#1}}

\newcommand\Diffpc{\textbf{Diff}_c}

\newcommand\PVar[1]{\mathbf{PVar}_{#1}}
\newcommand\PVarC{\mathbf{PVar}_{\CC}}
\newcommand\Sr[1]{\mathrm{S}{#1}}

\newcommand\Bordpo[1]{\mathbf{Bdp}_{{#1}}^{or}}
\newcommand\ClBordp[1]{\mathbf{l}\NBord{#1}{}{}}
\newcommand\CTub[3]{\mathbf{Tb}_{{#1 #3}}^{#2}}
\newcommand\CTubp[1]{\CTub{#1}{}{}}
\newcommand\CTubpp[1]{\mathbf{Tbp}_{#1}{}{}}
\newcommand\CTubppo[1]{\mathbf{Tbp}_{#1}^0}
\newcommand\CClose[3]{\mathbf{Cl}_{{#1 #3}}^{#2}}
\newcommand\CClosep[1]{\CClose{#1}{}{}}
\newcommand\Obj[1]{\mathrm{Obj}(#1)}
\newcommand\Mor[1]{\mathrm{Mor}(#1)}
\newcommand\Vect[1]{{#1}\textrm{-}\mathbf{Vect}}
\newcommand\Vecto[1]{{#1}\textrm{-}\mathbf{Vect}_0}
\newcommand\Mod[1]{{#1}\textrm{-}\mathbf{Mod}}
\newcommand\Modt[1]{{#1}\textrm{-}\mathbf{Mod}_t}
\newcommand\Rng{\mathbf{Ring}}
\newcommand\Grp{\mathbf{Grp}}
\newcommand\Grpd{\mathbf{Grpd}}
\newcommand\Grpdo{\mathbf{Grpd}_0}
\newcommand\HS[1]{\mathbf{HS}^{#1}}
\newcommand\MHS[1]{\mathbf{MHS}}
\newcommand\PHS[2]{\mathbf{HS}^{#1}_{#2}}
\newcommand\MHSq{\mathbf{HS}}
\newcommand\Sch{\mathbf{Sch}}
\newcommand\Sh[1]{\mathbf{Sh}\left(#1\right)}
\newcommand\QSh[1]{\mathbf{QSh}\left(#1\right)}
\newcommand\Var[1]{\mathbf{Var}_{#1}}
\newcommand\Varrel[1]{\mathbf{Var}_{#1}}
\newcommand\RVar[2]{\mathbf{Var}_{#1}/{#2}}
\newcommand\CVar{\Var{\CC}}
\newcommand\PHM[2]{\cM_{#1}^p(#2)}
\newcommand\MHM[1]{\cM_{#1}}
\newcommand\HM[2]{\textrm{HM}^{#1}(#2)}
\newcommand\HMW[1]{\textrm{HMW}(#1)}
\newcommand\VMHS[1]{VMHS({#1})}
\newcommand\geoVMHS[1]{VMHS_g({#1})}
\newcommand\goodVMHS[1]{\mathrm{VMHS}_0({#1})}
\newcommand\Par[1]{\mathrm{Par}({#1})}
\newcommand\K[1]{\mathrm{K}#1}
\newcommand\Ko[1]{\mathrm{K}{#1}_0}
\newcommand\KM[1]{\mathrm{K}\MHM{#1}}
\newcommand\KMo[1]{\mathrm{K}{\MHM{#1}}_0}
\newcommand\Ab{\mathbf{Ab}}
\newcommand\CPP{\cP\cP}
\newcommand\Bim[1]{{#1}\textrm{-}\mathbf{Bim}}
\newcommand\Span[1]{\mathbf{Span}({#1})}
\newcommand\ESpan[2]{\mathbf{Span}_{#2}({#1})}
\newcommand\Spano[1]{\mathbf{CoSpan}({#1})}
\newcommand\ESpano[2]{\mathbf{CoSpan}_{#2}({#1})}
\newcommand\GL[1]{\mathrm{GL}_{#1}}
\newcommand\SL[1]{\mathrm{SL}_{#1}}
\newcommand\PGL[1]{\mathrm{PGL}_{#1}}
\newcommand\AGL[1]{\mathrm{AGL}_{#1}}
\newcommand\Rep[1]{\mathfrak{X}_{#1}}
\newcommand\Repred[1]{\mathfrak{X}_{#1}^{r}}
\newcommand\Repirred[1]{\mathfrak{X}_{#1}^{ir}}
\newcommand\Charred[1]{\mathcal{R}_{#1}^{r}}
\newcommand\Charirred[1]{\mathcal{R}_{#1}^{ir}}
\newcommand\Repdiag[1]{\mathfrak{X}_{#1}^{ps}}
\newcommand\Repkappa[2]{\mathfrak{X}_{#1}^{#2}}
\newcommand\Reput[1]{\mathfrak{X}_{#1}^{ut}}

\newcommand\Qtm[1]{\mathcal{Q}_{#1}}
\newcommand\sQtm[1]{\mathcal{Q}_{#1}^0}
\newcommand\Fld[1]{\mathcal{F}_{#1}}
\newcommand\Id{\mathrm{Id}}

\newcommand\DelHod[1]{e\left(#1\right)}
\newcommand\RDelHod{e}
\newcommand\eVect{\mathcal{E}}
\newcommand\e[1]{\eVect\left(#1\right)}
\newcommand\intMor[2]{\int_{#1}\,#2}

\newcommand\Dom[1]{\mathcal{D}_{#1}}

\newcommand\Xf[1]{{X}_{#1}}					
\newcommand\Xs[1]{\mathfrak{X}_{#1}}							

\newcommand\Xft[2]{\overline{{X}}_{#1, #2}} 
\newcommand\Xst[2]{\overline{\mathfrak{X}}_{#1, #2}}			

\newcommand\Xfp[2]{X_{#1, #2}}			
\newcommand\Xsp[2]{\mathfrak{X}_{#1, #2}}						

\newcommand\Xfd[2]{X_{#1; #2}}			
\newcommand\Xsd[2]{\mathfrak{X}_{#1; #2}}						

\newcommand\Xfm[3]{\mathcal{X}_{#1, #2; #3}}		
\newcommand\Xsm[3]{X_{#1, #2; #3}}					

\newcommand\XD[1]{#1^{D}}
\newcommand\XDh[1]{#1^{\delta}}
\newcommand\XU[1]{#1^{UT}}
\newcommand\XP[1]{#1^{U}}
\newcommand\XPh[1]{#1^{\upsilon}}
\newcommand\XI[1]{#1^{\iota}}
\newcommand\XTilde[1]{#1^{\varrho}}
\newcommand\Xred[1]{#1^{r}}
\newcommand\Xirred[1]{#1^{ir}}

\newcommand\Char[1]{\cR_{#1}}
\newcommand\Chars[1]{\cR_{#1}}
\newcommand\CharW[1]{\mathscr{R}_{#1}}

\newcommand\Ch[1]{\textrm{Ch}\,{#1}}
\newcommand\Chp[1]{\textrm{Ch}^+{#1}}
\newcommand\Chm[1]{\textrm{Ch}^-{#1}}
\newcommand\Chb[1]{\textrm{Ch}^b{#1}}
\newcommand\Der[1]{\textrm{D}{#1}}
\newcommand\Dp[1]{\textrm{D}^+{#1}}
\newcommand\Dm[1]{\textrm{D}^-{#1}}
\newcommand\Db[1]{\textrm{D}^b{#1}}

\newcommand\Gs{\cG}
\newcommand\Gq{\cG_q}
\newcommand\Gg{\cG_c}
\newcommand\Zs[1]{Z_{#1}}
\newcommand\Zg[1]{Z^{gm}_{#1}}
\newcommand\cZg[1]{\cZ^{gm}_{#1}}

\newcommand\RM[2]{R\left(\left.#1\right|#2\right)}
\newcommand\RMc[3]{R_{#1}\left(\left.#2\right|#3\right)}
\newcommand\set[1]{\left\{#1\right\}}
\newcommand{\Stab}{\textrm{Stab}\,} 
\renewcommand{\tr}{\textrm{tr}\,}             
\newcommand\EuChS{E}             
\newcommand\EuCh[1]{E\left(#1\right)}             

\newcommand{\Ann}{\textrm{Ann}\,}
\newcommand{\Rad}{\textrm{Rad}\,}  
\newcommand\supp[1]{\mathrm{supp}{(#1)}}
\newcommand\coh[1]{\left[H_c^\bullet\hspace{-0.05cm}\left(#1\right)\right]}
\newcommand\Kclass[1]{\left[#1\right]}
\newcommand\Bt{B_t}
\newcommand\Be{B_e}
\newcommand\re{\textrm{Re}\,}
\newcommand\imag{\textrm{Im}\,}
\newcommand\Kahc{\textbf{K\"ah}_c}

\newcommand\Cone[1]{\textrm{Cone}\left({#1}\right)}
\newcommand\conic{\textrm{C}}
\newcommand\op[1]{{#1}^{op}}

\newcommand\Ccs[1]{C_{cs}(#1)}
\newcommand\can{\textrm{can}}
\newcommand\var{\textrm{var}}
\newcommand\Perv[1]{\textrm{Perv}(#1)}
\newcommand\DR[1]{\textrm{DR}{#1}}
\renewcommand\Gr[2]{\textrm{Gr}_{#1}^{#2}\,}
\newcommand\ChV{\textrm{Ch}\,}
\newcommand\VerD{^{\textrm{Ve}}\DD}
\newcommand\DHol[1]{\textrm{D}^b_{\textrm{hol}}(\cD_{#1})}
\newcommand\RegHol[1]{\Mod{\cD_{#1}}_{\textrm{rh}}}
\newcommand\Drh[1]{\textrm{D}^b_{\textrm{rh}}(\cD_{#1})}
\newcommand\Dcs[2]{\textrm{D}^b_{\textrm{cs}}({#1}; {#2})}
\newcommand{\rat}{\mathrm{rat}}
\newcommand{\dmod}{\mathrm{Dmod}}

\hyphenation{mul-ti-pli-ci-ty}

\hyphenation{mo-du-li}

\begin{abstract}
In this paper, we construct a lax monoidal Topological Quantum Field Theory that computes virtual classes, in the Grothendieck ring of algebraic varieties, of $G$-representation varieties over manifolds with conic singularities, which we will call nodefolds. This construction is valid for any algebraic group $G$, in any dimension and also in the parabolic setting. In particular, this TQFT allow us to compute the virtual classes of representation varieties over complex singular planar curves. In addition, in the case $G = \SL{2}(k)$, the virtual class of the associated character variety over a nodal closed orientable surface is computed both in the non-parabolic and in the parabolic scenarios. 
\end{abstract}
\null
\vspace{-1.1cm}
\maketitle

\vspace{-0.8cm}


\section{Introduction}
\let\thefootnote\relax\footnotetext{\noindent \emph{2020 Mathematics Subject Classification}. Primary:
  57R56, 
 Secondary:
 19E08, 
 32S50, 
 14D21, 
 20G05. 

\emph{Key words and phrases}: cone sigularities, character variety, TQFT, Grothendieck ring.}

Let $X$ be a topological space with finitely generated fundamental group and let $G$ be an algebraic group over an algebraically closed field $k$. The \emph{representation variety} of $X$ is the set of representations
$$
	\Rep{G}(X) = \left\{\rho: \pi_1(X) \to G\right\}.
$$
This set inherits a natural structure of algebraic variety from the one of $G$. Moreover, there is an action of $G$ on $\Rep{G}(X)$ by conjugation, which corresponds to identifying isomorphic representations. This action is far from being free and, indeed, it has varying stabilizers on the subvariety of reducible representations according to the Jordan-Holder filtration of the $k[G]$-module. For this reason, the orbit space $\Char{G}(X) = \Rep{G}(X) / G$ is no longer an algebraic variety. To overcome this issue, if $G$ is reductive, we can consider the Geometric Invariant Theory (GIT) quotient
$$
	\Char{G}(X) = \Rep{G}(X)\sslash G,
$$
usually referred to as the \emph{character variety} of $X$ into $G$ also known as  the moduli space of isomorphism classes of representations of $\pi_1(X)$ into $G$.


Character varieties have been widely studied due to their tight relation with other important gauge theoretic moduli spaces, in a celebrated correspondence known as the non-abelian Hodge correspondence. Through this correspondence, if $G = \GL{n}(\CC)$ (resp.\ $G = \SL{n}(\CC)$) and $X$ is a smooth projective algebraic curve, then $\Char{G}(X)$ turns out to be isomorphic to the moduli space of flat connections on rank $n$ vector bundles (resp.\ and fixed determinant) \cite{SimpsonI,SimpsonII} and diffeomorphic to the moduli space of degree $0$ and rank $n$ Higgs bundles (resp.\ and fixed determinant) \cite{Corlette:1988,Simpson:1992}. Thanks to these maps, $\Char{G}(X)$ is naturally endowed with a non-trivial hyperk\"ahler structure \cite{hitchin1992hyper, Hitchin}.
 
Ever more striking correspondences can be obtained if we endow the character variety with a parabolic structure $Q$. This is a finite collection of punctures $p_1, \ldots, p_s \in X$ that we remove from $X$ together with a collection of $G$-conjugacy classes $\lambda_1, \ldots, \lambda_s \subseteq G$, usually referred to as the \emph{holonomies}. In this way, the \emph{parabolic representation variety}, $\Rep{G}(X,Q)$, is the collection of representations $\rho: \pi_1(X - \left\{p_1, \ldots, p_s\right\}) \to G$ such that if $\gamma_i$ is the positively oriented loop around $p_i$ then $\rho(\gamma_i) \in \lambda_i$ (see Section \ref{sec:representation-varieties} for a precise definition). Analogously, the \emph{parabolic character variety} is $\Char{G}(X,Q) = \Rep{G}(X, Q) \sslash G$. The parabolic structure allows us to bend the non-abelian Hodge correspondence in such a way that, for some holonomies, $\Char{G}(X,Q)$ is isomorphic to the moduli space of rank $n$ flat connections with logarithmic singularities around $p_1, \ldots, p_s$ (and the monodromy around these points is determined by the holonomies of $Q$) and diffeomorphic to the moduli space of rank $n$ parabolic Higgs bundles (and the weights of the parabolic structure are determined by the holonomies of $Q$) \cite{Simpson:parabolic}.

For these and many more reasons, the topology and algebraic structure of character varieties have been objective of intense study. Particularly interesting is to determine their class in the Grothendieck ring of algebraic varieties, $[\Char{G}(X)] \in \K{\Var{k}}$, also known as the virtual class. This problem has been addressed from different points of view. In the seminal paper \cite{Hausel-Rodriguez-Villegas:2008}, the authors proposed a method for computing the $E$-polynomial of $\Char{G}(X)$ (a slightly coarser invariant than the virtual class) by means of counting points of the character variety on finite fields, in the spirit of the Weil conjectures. This work has been subsequently extended in \cite{Hausel-Letellier-Villegas,mellit2020poincare,Mereb} to the parabolic generic setting and other groups. 
Another approach was initiated in \cite{LMN}, based on looking for geometric hands-on decompositions of the character variety according to the stabilizers of the adjoint action. This work was extended in \cite{LM,martinez2016polynomials} to give rise to very explicit formulas for the $E$-polynomial in the case of rank $2$.

Finally, a third approach was initiated in \cite{GPLM-2017}, based on the construction of a Topological Quantum Field Theory (TQFT) that computes the virtual classes of representation varieties over closed manifolds of arbitrary dimension into any group. Using this techique, the virtual classes of parabolic character varieties where computed for $G = \SL{2}(k)$ in \cite{GP-2018} for Jordan type holonomies and in \cite{gonzalez2020virtual} for arbitrary holonomies. This method has also been used for automatizing the calculations and reaching higher rank for $G$ the subgroup $\mathbb{U}_r \subseteq \GL{r}(k)$ of upper-triangular matrices for $r = 2,3,4$, both for orientable \cite{hablicsek2020virtual} and non-orientable surfaces \cite{vogel2020representation}.

However, very few works have addressed the problem of computing the virtual class of character varieties over topological spaces different than closed surfaces or closed surfaces with punctures. One of the most studied cases arises when we consider an embedded link $K \subseteq \RR^3$ and we take $X=\RR^3-K$. These are the so-called knot character varieties and have been studied, for instance, in \cite{2006.01810,Munoz:2009,Munoz-Lawton:2016} for torus knots or in \cite{heusener2016sl} for the figure eight knot. In addition, in \cite{florentino2019polynomials} the $E$-polynomial is computed for the complement of unknoted links (i.e.\ free fundamental group), and in \cite{florentino2017hodge} for tori $X = S^1 \times \ldots \times S^1$ (i.e. free abelian fundamental group).
 
In this paper, we will study character variety over singular manifolds. Specifically, we will focus on an analog of topological manifolds that we will call \emph{nodefolds}, in which nodal singularities are allowed. More precisely, instead of being locally euclidean, as  topological manifolds are, the nodefolds are locally a cone over a closed manifold. In particular, these nodefolds subsume topological manifolds since discs, the local models of topological manifolds, can be seen as cones over the sphere.

This notion encompasses much more general spaces, like plane complex projective curves (possibly singular) and, in general, complex hypersurfaces. Indeed, classical results about the structure of the Milnor fibration \cite{milnor2016singular} show that, from a purely topological point of view, a singularity in a plane curve is either nodal-like or cusp-like. In the former case, locally the curve looks like a cone over a bunch of circles, that we call bouquet of cones, which justifies the name \emph{nodefold}. Notice that, in the later case of a cusp-like singularity, the singular point is non-smooth but topologically the curve is locally euclidean (the homeomorphism is, of course, non-differentiable). 
In this vein, cusp-like singularities remain undetected through the eyes of the fundamental group of the curve and, eventually, the character variety is the same as if the curve was non-singular.

The study of moduli spaces of Higgs bundles and flat connections over singular curves is a very active reseach area. In \cite{bhosle1995representations}, it is provided a very precise description of those generalized parabolic bundles associated to torsion free sheaves over a nodal curve that come from a representation of its fundamental group. In particular, it is shown that the analog of the Narasimhan-Seshadri theorem \cite{narasimhan1965stable} in the singular setting does not hold, since there exist (generalized) parabolic bundles not associated to any representation. Similar results exist regarding representations of the fundamental group and Higgs bundles over nodal curves, as studied in \cite{bhosle2014hitchin}. Nevertheless, there is still some hope in recovering a sort of non-abelian Hodge correspondence in the nodal setting if appropriate framings for the flat vector bundle on the singular points are fixed, as sketched to some extent in \cite{bhosle2013grassmannian} but the problem is still open. For a survey regarding this interplay on nodal curves, see \cite{logares2019higgs}.

Despite of these thorough studies, very few is known about the algebraic structure of character varieties over singular curves themselves since they are always studied through moduli spaces of holomorphic and Higgs bundles. 
In order to address this problem, in this paper, we extend the construction of the TQFT of \cite{GPLM-2017} to representation varieties over nodefolds. More precisely, let $\NBordnp{n}$ be the category of $n$-dimensional nodefold bordisms with a finite number of basepoints, that is, the analogy of the usual category of smooth bordisms but on which conic singularities are allowed. Also, let $\Mod{\K{\Var{k}}}$ the usual category of modules over the ground ring $\K{\Var{k}}$. We will prove the following result.

\begin{theorem*}
Let $G$ be an algebraic group. For any $n \geq 1$, there exists a lax monoidal symmetric functor
$$
	Z_{G}: 	\NBordnp{n} \to \Mod{\K{\Var{k}}}
$$
such that, for any closed connected $n$-nodefold $X$ and any basepoint $\star \in X$, we have $Z_G(X, \star)(1)=[\Rep{G}(X)]$, where $1 = \in \K{\Var{k}}$ is the unit of the ring.
\end{theorem*}

Moreover, this construction extends naturally to the parabolic setting. In this framework, we equip bordisms with parabolic structures. In dimension $2$, as aforementioned, the parabolic structure is given by a collection of punctures and holonomies, but in higher dimension puctures should be replaced by codimension $2$ cooriented submanifolds with a choice of holonomy for the meridians around them (see Section \ref{sec:representation-varieties} for further details). In this way, we also obtain a functor $Z_{G}: \NBordnp{n}(\Lambda) \to \Mod{\K{\Var{k}}}$, where $\NBordnp{n}(\Lambda)$ is the category of nodefolds bordisms with parabolic data with holonomies in $\Lambda$.

Furthermore, we will show that this functor is actually a functor of $2$-categories in a natural way. Suppose that $X, X'$ are two $n$-dimensional nodefolds (i.e.\ two morphisms of $\NBordnp{n}$). We will say that a map $f: X \to X'$ is a \emph{conic degeneration} if, roughly speaking, it is a homeomorphism outside the singular points of $X$ and $X'$. It captures the idea of ``collapsing points'' in a nodefold, in which several points of $X$ are collapsed via $f$ into a single point of $X'$ in such a way that $f$ becomes a ramified covering. In this manner, the category $\NBordnp{n}$ is endowed with a $2$-category structure by taking the conic degenerations a $2$-morphisms. As we will see, these $2$-morphisms in $\NBordnp{n}$ turns into a natural $2$-structure in $\Mod{\K{\Var{k}}}$ whose $2$-morphisms between homomorphisms are called \emph{twists}. This $2$-structure arose naturally in the original construction \cite{GPLM-2017} but, as we will see, in the singular setting it plays a much important role since it controls the effect of conic degenerations on the representation variety in a very precise way.

Degenerations arise naturally in moduli space theory. An important instance of this is the theory of $\lambda$-connections on an algebraic curve $X$, as developed in \cite{SimpsonI}. This $\lambda$-connections give rise to a family of moduli spaces $\cM_\lambda$ with $\lambda \in \CC$ such that for $\lambda = 1$ we recover the moduli space of flat connections and for $\lambda = 0$ we obtain the moduli space of Higgs bundles. 

In the context of representation varieties over nodefolds, conic degenerations will play a very important role to relate the virtual class of a representation variety over node-surfaces (i.e.\ nodefolds of dimension $2$) and over smooth surfaces. Given a node-surface $X$, a \emph{normalization} is a conic degeneration $f: \Sigma \to X$, where $\Sigma$ is a smooth surface. Normalizations always exists and are unique up to homeomorphism. As byproduct of analysis of conic degenerations through the TQFT, we will obtain the following.

\begin{theorem*}
Let $G$ be an algebraic group. Let $X$ be a node-surface with $l$ singular points which, locally, are cones over bunches of $r_1, r_2, \ldots, r_s$ circles, respectively; and let $Q$ be a parabolic structure on $X$. Then, if $f: \Sigma \to X$ is the normalization of $X$, and $Q'$ is the parabolic structure induced in $\Sigma$ via $f$, we have
$$
	[\Rep{G}(X, Q)] = [\Rep{G}(\Sigma, Q')] \times [G]^{r_1 + \ldots + r_l}.
$$
\end{theorem*}

Of course, this result can be obtained more easily through the usual Seifert-Van Kampen theorem and Hurwitz formula for ramified coverings. However, philosophically, the proof provided in this work provides a deeper insight into the structure of the representation variety than the usual proof. Seifert-Van Kampen theorem (in its version for fundamental groupoids) actually underlies the whole construction of the TQFT so, in some sense, the proof by means of the TQFT and conic degenerations encodes all the needed topological ingredients in a very effective way. In particular, the same kind of techniques used in this paper could be used to understand the effect of wilder singularities on nodefolds (namely, singularities arising in complex hypersurfaces) at the level of the representation variety. We postpone this analysis to a future work.

Finally, in this paper we will study how to transfer the results about the $\SL{r}(k)$-virtual classes of representations varieties over node-surfaces into the ones for the associated character varieties. For this purpose, we will use the theory of pseudo-quotients, as developed in \cite{GP-2018b}. This theory allows us to chop the representation variety into strata on which the action can be understood more easily. This chopping is not, in general, compatible with the GIT quotient but it is compatible with a weaker version called pseudo-quotient, as introduced in \cite{GP-2018b}. Pseudo-quotients may not agree with GIT quotients on the nose, but their virtual classes agree. In this way, we can reassemble the virtual class of the quotient just by adding up the virtual classes of the quotients of each strata.

In our case, roughly speaking the idea is to decompose the representation variety into its loci of reducible and irreducible representations as
$$
	\Rep{\SL{r}(k)}(X) = \Repred{\SL{r}(k)}(X) \sqcup \Repirred{\SL{r}(k)}(X).
$$

On $\Repirred{G}(X)$, the stabilizer of the action is the set of multiples of the identity, so the action of $\PGL{r}(k)$ is free there. In this way, the GIT quotient on $\Repirred{\SL{r}(k)}(X)$ reduces to an orbit space and, thus $[\Repirred{\SL{r}(k)}(X) \sslash \SL{r}(k)] = [\Repirred{\SL{r}(k)}(X)] / [\PGL{r}(k)]$. On the other hand, each orbit of the reducible locus, $\Repred{\SL{r}(k)}(X)$, is close to the orbit of an semi-simple representation (a situation called a core). In this way, the GIT quotient $\Repred{\SL{r}(k)}(X) \sslash \SL{r}(k)$ is equivalent to a product of lower rank representations, and the action reduces to the action of the symmetric group permuting representations of the same rank.

In particular, in the rank $2$ case, we have that the semi-simple representations of representations of $\Repred{\SL{2}(k)}(X)$ are just the diagonal representations, and this diagonal form is unique up to permutation of the eigenvalues. Moreover, due to the unit determinant condition, the two eigenvalues are inverse, so we finally get that 
$$
	\Repred{\SL{2}(k)}(X) = (k^*)^N/\ZZ_2.
$$
Here, $N$ is a positive integer that depends on the number of generators in a presentation of $\pi_1(X)$ and the number of punctures on the parabolic structure $Q$, and the action of $\ZZ_2$ is given by $(\lambda_1, \ldots, \lambda_N) \mapsto (\lambda_1^{-1}, \ldots, \lambda_N^{-1})$. Therefore, using \cite[Theorem 4.5]{GP-2018b} we get that
\begin{align*}
	\left[\Char{\SL{2}(k)}(X)\right] &= \left[\Repred{\SL{2}(k)}(X) \sslash \SL{2}(k)\right] + \left[\Repirred{\SL{2}(k)}(X) \sslash \SL{2}(k)\right] \\
	&= \left[(k^*)^N/\ZZ_2\right] + \frac{\left[\Repirred{\SL{2}(k)}(X)\right]}{\left[\PGL{2}(k)\right]} = \left[(k^*)^N/\ZZ_2\right] + \frac{\left[\Rep{\SL{2}(k)}(X)\right] - \left[\Repred{\SL{2}(k)}(X)\right]}{\left[\PGL{2}(k)\right]} .
\end{align*}
Notice that the last equality is particularly useful, since the computation of $[\Repirred{\SL{2}(k)}(X)]$ may be quite involved since it is an open set of the whole character variety. However, the the virtual class of its complement, $\Repred{\SL{2}(k)}(X) \subseteq \Rep{\SL{2}(k)}(X)$, is much easier to get since it is a very small closed set whose counting reduces to lower rank cases. Since the virtual class of the whole variety is obtained through the TQFT, we have all the needed ingredients.

Analogous arguments can be carried out in the parabolic setting, but now taking into account that the holonomies of the punctures may restrict the possible reducible representations. Performing this analysis carefuly, we obtain the main result of this paper.

\begin{theorem*}
Let $X$ be the closed node-surface with $l$ singular points with a total of $b = r_1 + r_2 + \ldots + r_l$ branches, and whose normalization is a closed orientable genus $g$ surface. Fix arbitrary eigenvalues $\lambda_1, \ldots, \lambda_s \in k-\set{0, \pm 2}$. Let $\alpha_+$ (resp.\ $\alpha_-$) be one half of the number of tuples $(\epsilon_1, \ldots, \epsilon_s) \in \set{\pm 1}^s$ such that $\lambda_1^{\epsilon_1}\cdots \lambda_s^{\epsilon_s} = 1$ (resp.\ such that $\lambda_1^{\epsilon_1}\cdots \lambda_s^{\epsilon_s} = -1$). Let $Q$ be a parabolic structure with $r$ punctures with holonomy of Jordan type with trace $2$ and $s>0$ punctures with diagonalizable holonomies with eigenvalues $\lambda_1, \ldots, \lambda_s$. Denote $q = [k] \in \K{\Var{k}}$ the virtual class of the affine line. The virtual class of the $\SL{2}(k)$-character variety $\Char{}(X,Q) = \Char{\SL{2}(k)}(X,Q)$ in the localization of $\K{\Var{k}}$ by $q, q+1, q-1$ is
\begin{itemize}
	\item If $s = r = 0$, then
\small
\begin{align*}
	\left[\Char{}(X,Q)\right] &= \frac{1}{2 q^3 \left(q^2-1\right)^3}\left[q^8 \left((q-1)^b+1\right) (q-1)^{2 g}+q^7 \left((q-1)^b+1\right) (q-1)^{2 g}\right.\\
	&-2 q^6 \left((q-1)^b+1\right) (q-1)^{2 g}-2 q^5
   \left((q-1)^b+1\right) (q-1)^{2 g}\\
   &+q^4 \left((q-1)^b+1\right) (q-1)^{2 g}+q^3 \left((q-1)^b+1\right) (q-1)^{2 g}\\
   &+(q-1)^3 q^3
   \left((q+1)^b-1\right) (q+1)^{2 g+2}+\left(q \left(q^2-1\right)\right)^b \left((q+1)^2 (q-1)^{2 g} \left(4^g+q-3\right) q^{2g} \right.\\
   &\left.+(q-1)^2 (q+1)^{2 g} \left(4^g+q-1\right) q^{2 g}+2 \left(q^2-1\right)^{2 g} \left(q^{2 g}+q^2\right)\right)\\
   &\left.-\left(q^2-1\right)q^{2 g+1} \left(2 \left(4^g \left(q^4+q^2+1\right)+(q+1)^2 (q-1)^{2 g+1}\right)-q^2 3\cdot 2^{2 g+1}\right)\right]
		\end{align*}
\normalsize
\item If $s = 0$ and $r > 0$, then
\small
\begin{align*}
\left[\Char{}(X, Q)\right] &=\frac{q^{2 g-3}}{2 \left(q^2-1\right)^3}\left[\left(4^g-3\right) \left(q \left(q^2-1\right)\right)^b (q-1)^{2 g+r}\right.\notag\\
&\left.+\left(q \left(q^2-1\right)\right)^b \left(q \left(4^g(q+2)+q^2-q-5\right) (q-1)^{2 g+r}+2 \left(q^2-1\right)^{2 g+r}\right.\right.\notag\\
&\left.\left.+q (-1)^r \left(4^g (q-2)+q^2-3 q+3\right) (q+1)^{2g+r}+\left(4^g-1\right) (-1)^r (q+1)^{2 g+r}\right)\right.\\
&\left.+\left(q^2-1\right) (-1)^r
   2^{b+2 g} q^b \left(q^5-q^4+q^3
  +\left(q^2-1\right)^2 (1-q)^r+2 q^2+q-1\right)\right.\notag\\
  &\left. +\left(q^2-1\right) (-1)^{r+1} q^{b+3} 3\cdot 2^{b+2 g}\right]\notag
\end{align*}
\normalsize
	\item If $s > 0$ and $r = 0$, then
\small
\begin{align*}
	\left[\Char{}(X, Q)\right] &= (q-1)^{2 g-2} \left(\frac{4 (q-1)^b \alpha _+ \left(q^3-q^{b+2 g+s}\right)}{q^3}\right.\notag\\
	&\left.+(q-1) (q+1) \left(q \left(q^2-1\right)\right)^{b-2}
   q^{2 g+s-1} \left(\left(\frac{q}{q+1}\right)^{-2 g-s+2}\right.\right.\\
   &\left.\left.+(q+1)^{2 g+s-2}+2^{2 g+s-1}-2^s\right)-2 (q-1)^b \alpha _++(q-1)^{b+1} \alpha
   _+\right)\notag\\
   &+ {\cI}_0(\xi_1, \ldots, \xi_s)(q^3-q)^{b-2}\notag
\end{align*}
\normalsize
where the interaction term is given by
\small
\begin{align*}
{\cI}_0(\xi_1, \ldots, \xi_s) =\, & q^{s-1}(q - 1)^{2g-1}(q+1)(\alpha_+ + \alpha_-)\left(q(q + 1)^{2g-1} + q^{2g}(q + 1)^{2g-1}\right.
\\ 
&\left.- q^{2g}(q + 1)^{2g-1} - q(q + 1)^{2g-1}\right) +q^{2g+s-1}(q - 1)^{2g}(q + 1)\alpha_+.
\end{align*}
\normalsize
	\item If $s, r>0$, then
\small
\begin{align*}
	\left[\Char{}(X, Q)\right] &=  q^{2g + s-2}(q - 1)^{2g + r-2}\left(2^{2g+s-1} - 2^s + (q + 1)^{2g + r+s-2} \right)(q^3-q)^{b-1} \\
		& \,\,\,\,\,\,\,+ {\cI}_r(\xi_1, \ldots, \xi_s)(q^3-q)^{b-2},
\end{align*}
\normalsize
where the interaction term is given by
\small
\begin{align*}
{\cI}_r(t_1, \ldots, t_s) =\, & q^{2g + s-1}(q - 1)^{2g + r-1}(\alpha_+ + \alpha_-)\bigg( 2^{2g} + 2^{2g}q  - 2q -2 \\
& \left. + (q + 1)^{2g + r} + (q+1)\left(1 - 2^{2g-1} - \frac{1}{2}(q + 1)^{2g + r-1}\right)\right) \\
& + q^{2g + s-1}(q - 1)^{2g + r}(q+1)\alpha_+.
\end{align*}
\normalsize
\end{itemize}
\end{theorem*}

The remaining cases can be easily reduced to these versions, as explained in Section \ref{sec:parabolic-arbitrary}.

The structure of this paper is as follows. In Section \ref{sec:nodefols} we introduce nodefolds and conic degenerations, as well as representation and character varieties over them and their parabolic counterparts. Section \ref{sec:tqft} is devoted to the construction of the lax monoidal TQFT computing virtual classes of representation varieties over nodefolds. In particular, in Section \ref{sec:conic-deg-tqft} we shall exploit the properties of conic degenerations under the built TQFT to relate the virtual classes of representation varieties over a nodefold and over its normalization. Section \ref{sec:charvar-nodefolds} is focused on the analysis of the GIT quotient to get the virtual class of the associated character variety in the $G=\SL{2}(k)$ case. Section \ref{sec:charvar-nodefolds-nopar} is devoted to the non-parabolic case, while Sections \ref{sec:parabolic-jordan} and \ref{sec:parabolic-diagonal} focuses on the parabolic setting, with holonomies of Jordan and semi-simple type, respectively. Finally, in Section \ref{sec:parabolic-arbitrary} we show how these results can be put together to obtain the virtual class of parabolic $\SL{2}(k)$-character varieties with an arbitrary parabolic structure.

\subsection*{Acknowledgements}

The authors want to thank Carlos Florentino, Vicente Mu\~noz and Jaume Silva for very useful conversations. The first author acknowledges the hospitality of the Instituto de Ciencias Matem\'aticas in which part of this work was completed during a postdoctoral stay.

The first author has been partially supported by Severo Ochoa excellence programme SEV-2015-0554 and Spanish Ministerio de Ciencia e Innovaci\'on project PID2019-106493RB-I00.

\section{Nodefolds and representation varieties}\label{sec:nodefols}

In the following we introduce a new class of topological spaces that look like topological manifolds but with some special points around which they are no longer locally euclidean but locally a cone or even a bouquet of cones as in Figure \ref{img:cones}. We shall call this type of topological spaces nodefolds because they will serve us to describe Riemann surfaces corresponding to nodal curves. \footnote{The reader may think that other names like \textit{conifold} would be more appropriate than \textit{nodefolds}. The reader is right. However, sadly the term conifold is yet in use for a similar, but not equal, concept in the context of string theory \cite{candelas1990rolling} so we will not use it here to avoid confusions.}

\begin{defn}
A \emph{nodefold} $X$ is a second countable, paracompact topological space that is locally conic. The later means that, for any $p \in X$, there exists an open neighbourhood $U \subseteq X$, a non-empty closed manifold $B$ and a homeomorphism $\varphi: U \to \Cone{B}$ with $\varphi(p)=v_{B}$. Here
$$
	\Cone{B} = \left(B \times [0,1)\right)/(B \times \left\{0\right\})
$$
denotes the (topological) cone over $B$ and $v_{B} = [B \times \left\{0\right\}] \in \Cone{B}$ is the vertex of the cone. The dimension of the nodefold is then $\dim X = \dim B +1$.
\end{defn}


\begin{ex}
\begin{itemize}
	\item Observe that, if $B = S^{n-1}$, then $\Cone{S^{n-1}}$ is homeomorphic to the $n$-dimensional open ball. In that case, we will say that the point $p$ with $\varphi(p)=v_{B}$, is a \emph{smooth point}. Otherwise, we will call $p$ a \emph{conic point}. In this manner, all the (topological) manifolds have a natural nodefold structure. Moreover, observe that all the points of $\Cone{B} - \left\{v_B\right\}$ are smooth so the set of conic points of a nodefold $X$ is a discrete set, denoted $\conic X$.
	\item On the other hand, $B$ may not be connected with $r$ connected components. In that case, $\Cone{B}$ is genuinelly a cone with $r$ branches. For instance, if $B$ is the disjoint union of $r$ copies of $S^{n-1}$, then $\Cone{B}$ is the result of gluing $r$ cones along their vertices, as shown in Figure \ref{img:cones}.
	
\begin{figure}[h]
\includegraphics[width=11cm]{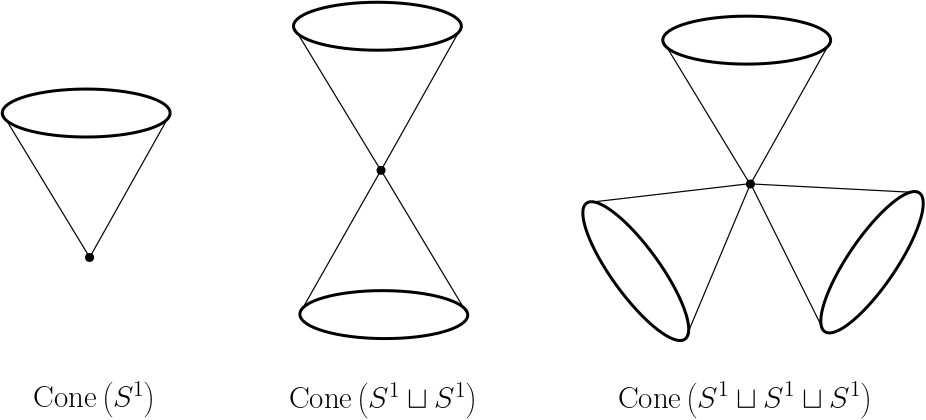}
\caption{Several cones over disjoint unions of $S^1$.}
\label{img:cones}
\end{figure}

\end{itemize}
\end{ex}

Analogously, a \emph{nodefold with boundary} $X$ is a second countable, paracompact topological space such that each point is either locally conic or a boundary point (i.e.\ locally homeomorphic to $\RR^n_+ = \left\{(x_1, \ldots, x_n) \in \RR^n\,|\, x_n \geq 0\right\}$). The boundary points of $X$ will be denoted by $\partial X$. Observe that $\partial X$ is a compact $(n-1)$-dimensional manifold. In this manner, all the conic points of $X$ belong to the interior $X - \partial X$.

\begin{defn}
Let $X_1$ and $X_2$ be nodefolds and let $f: X_1 \to X_2$ be a continuous map. We will say that $f$ is a \emph{conic degeneration} if there exists a finite subset $S \subseteq \conic X_2$ such that $f^{-1}(s)$ is finite for all $s \in S$ and $f: X_1 - f^{-1}(S) \to X_2 -S$ is a homeomorphism.
\end{defn}

\subsection{Representation varieties}\label{sec:representation-varieties}

As with usual topological manifolds, we can consider representation varieties over a nodefold. Fix an algebraic group $G$ over a certain algebraically closed field $k$. Suppose that $X$ is a compact nodefold (maybe with boundary). Then, the representation variety over $X$, $\Rep{G}(X)$, is the set of group representations
$$
	\rho: \pi_1(G) \to G.
$$
Recall that this set is naturally endowed with the structure of an algebraic variety as follows. Since $X$ is compact, we have a presentation of $\pi_1(G)$ with finitely many generators $\pi_1(G) = \langle \gamma_1, \ldots, \gamma_n \,|\, R_\alpha(\gamma_1, \ldots, \gamma_n) = 1\rangle$. This defines an injective map $\Psi: \Rep{G}(X) \to G^n$ given by $\Psi(\rho) = (\rho(\gamma_1), \ldots, \rho(\gamma_n))$. The image of $\Psi$ is an algebraic set of $G^n$ so we can put in $\Rep{G}(X)$ the algebraic structure induced by $\Psi$. In other words, we have
$$
	\Rep{G}(X) = \left\{(g_1, \ldots, g_n) \in G^n \,|\, R_\alpha(g_1, \ldots, g_n) = 1\right\}.
$$ 

Furthermore, we can extend this definition to consider several basepoints. Let $A \subseteq X$ be a finite set and let $\Pi(X,A)$ be the fundamental groupoid of $X$ with basepoints in $X$. The $G$-representation variety of $(X,A)$, $\Rep{G}(X,A)$, is the set of groupoid homomorphisms $\rho: \Pi(X,A) \to G$. It can be reduced to usual representation varieties as follows. Pick points $a_1, \ldots, a_m \in A$ in different connected components of $X$ and such that any other point of $A$ shares connected component with some of the points $a_i$. In this way, a groupoid representation $\rho: \Pi(X,A) \to G$ is completelly determined by the induced group representations $\pi_1(X, a_i) \to G$ together with the images of any $|A|-m$ joining the points of $A$ with the basepoint in the same component. These later images are unrestricted, so they can be any element of $G$, which gives rise to a natural identification
\begin{equation}\label{eq:rep-variety-general}
	\Rep{G}(X, A) = \prod_{i=1}^m \Rep{G}(X, a_i) \times G^{|A|-m}.
\end{equation}
Observe that each of the factors $\Rep{G}(X, a_i)$ is an usual representation variety, so $\Rep{G}(X, A)$ is naturally endowed with the structure of an algebraic variety too. For further details check \cite{GP-2018,GPLM-2017}

This construction inherits the functoriality properties from the fundamental groupoid. In this way, if $f: (X, A) \to (X', A')$ is a continuous map with $f(A) \subseteq A'$, it induces a regular morphism
$$
	\Rep{G}(f): \Rep{G}(X', A') \to \Rep{G}(X, A).
$$
In particular, if $X \subseteq X'$, the inclusion induces a restriction map $\Rep{G}(X', A') \to \Rep{G}(X, A' \cap X)$.

Finally, we can also consider parabolic structures on the representation variety. In analogy with \cite{GP-2018}, a parabolic structure on a nodefold $X$ (maybe with boundary) with basepoints $A$ is a finite set $Q = \left\{(S_1, \lambda_1), \ldots, (S_r, \lambda_r)\right\}$ such that
\begin{itemize}
	\item $S_i \subseteq X$ are pairwise disjoint co-oriented smooth submanifolds of codimension $2$ with $S_i \cap \conic X = S_i \cap A = \emptyset$ and such that all the submanifolds $S_i$ intersect $\partial X$ transversally.
	\item $\lambda_i \subseteq G$ are algebraic subvarieties that are invariant under the action of $G$ by conjugation. Typically, we will take $\lambda_i = [g_i]$ the conjugacy classes of some elements $g_i \in G$. 
\end{itemize}

\begin{rmk}
If $Q$ is a parabolic structure on a nodefold $(X, A)$ with boundary, there is naturally induced parabolic structure on $\partial X$, denoted $Q|_{\partial X}$, given by the collection of pairs $(S_i \cap \partial X, \lambda_i)$ for $(S_i, \lambda_i) \in Q$ with $S_i \cap \partial X \neq \emptyset$. Observe that it is crucial at this point that $S_i$ intersects $\partial X$ transversally, as imposed above. In a similar vein, if $M \subseteq \partial X$ is the disjoint union of some connected components of $\partial X$, we can also consider $Q|_M$.
\end{rmk}

Let us denote by $S = \bigcup_i S_i$ the total collection of submanifolds of the parabolic structure $Q$. Given a loop $\gamma \in \Pi(X - S, A)$, we will say that $\gamma$ is \emph{around $S_i$} if $\gamma \in \Ker  \iota_i$, where $\iota_i: \Pi(X - S, A) \to \Pi(X - (S - S_i), A)$ is the induced map by the inclusion. These loops can be easily described as those lying in the complement of the zero section of the normal bundle $\nu S_i$ to $S_i$ (embedded into $X$ as a tubular neighbourhood) that vanish when the zero section is added.

In other words, as proven in \cite{Smith:1978}, the loops around $S_i$ are the normal subgroup generated by any meridian (as normal subgroup). The meridians are, given a point $x \in S_i$, the usual generators of the fundamental group of the fiber $\nu_x S_i - \left\{0\right\} \sim S^1$ (see Figure \ref{img:meridian}). With this description, we will say that a meridian around $S_i$ is positive if it is positively oriented with respect to the orientation of the holed plane $\nu_x S_i - \left\{0\right\}$ given by the co-orientation of $S_i$.

\begin{figure}[h]
\includegraphics[width=5.5cm]{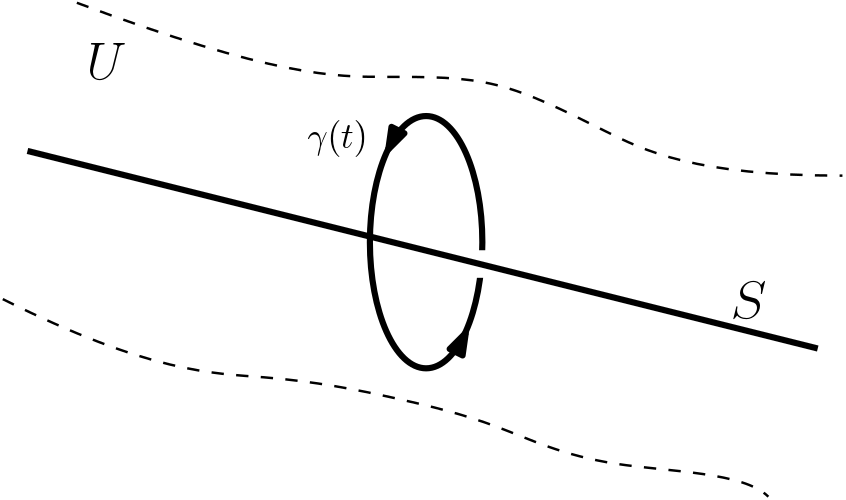}
\caption{Meridian around a codimension $2$ submanifold $S$.}
\label{img:meridian}
\end{figure}

Given such parabolic structure $Q$, we can consider the associated parabolic $G$-representation variety, denoted $\Rep{G}(X, A, Q)$. Then, $\Rep{G}(X, A, Q)$ is the set of groupoid homomorphisms $\rho: \Pi(X - S, A) \to G$ such that $\rho(\gamma) \in \lambda_i$ if $\gamma$ is a positive meridian around $S_i$. Again, choosing wisely the generators of $\Pi(X, A)$ (see \cite[Section 4]{GP-2018}), it is possible to give a natural decomposition of $\Rep{G}(X, A, Q)$ into a product of algebraic varieties, in the spirit of equation (\ref{eq:rep-variety-general}). In this manner, $\Rep{G}(X, A, Q)$ is also an algebraic variety.

\begin{rmk}
Observe that, since the collection of loops around $S_i$ are the normal subgroup generated by any positive meridian $\gamma$, the condition $\rho(\gamma) \in \lambda_i$ actually restricts the image of all these loops. Moreover, any two meridians are conjugated so, since the subvariety $\lambda_i$ is closed by conjugation, the parabolic condition does not depend on the chosen meridian.
\end{rmk}
 
Finally, suppose that $f: (X, A) \to (X', A')$ is a conic degeneration. Then, given a parabolic structure $Q' = \left\{(S_i', \lambda_i)\right\}$ in $(X',A')$, $f$ induces a parabolic structure $f^*Q = \left\{(f^{-1}(S_i), \lambda_i)\right\}$ in $(X, A)$. Moreover, if $Q$ is another parabolic structure on $(X, A)$ such that $Q \subseteq f^*Q$, the induced map of representation varieties actually preserves the parabolic structures giving rise to a map
$$
	\Rep{G}(f): \Rep{G}(X', A', Q') \to \Rep{G}(X, A, Q). 
$$
 
\subsection{Gluing nodefolds}

In this section, we discuss some gluing properties of nodefolds along boundaries and how do they affect to the associated representation varieties. In analogy with \cite{GPLM-2017}, this will allow us to define the field theory part of a TQFT (see Section \ref{sec:tqft}).

Consider two nodefolds with boundary $X$ and $X'$ and suppose that their boundaries share a common component, that it we can decompose the boundaries as $\partial X = M_1 \sqcup M_2$ and $\partial X' = M_2 \sqcup M_3$, as depicted in Figure \ref{img:gluing}. We can use this common boundary to glue $X$ and $X'$ together to give rise to a topological space $X \cup_{M_2} X'$. It is straightforward to check that $X \cup_{M_2} X'$ actually inherits a nodefold structure, and that the conic points satisfy $\conic\left(X \cup_{M_2} X'\right) = \conic X \sqcup \conic X'$. Moreover, the new nodefold has boundary $\partial \left(X \cup_{M_2} X'\right) = M_1 \sqcup M_3$.

\begin{figure}[h]
\includegraphics[width=12cm]{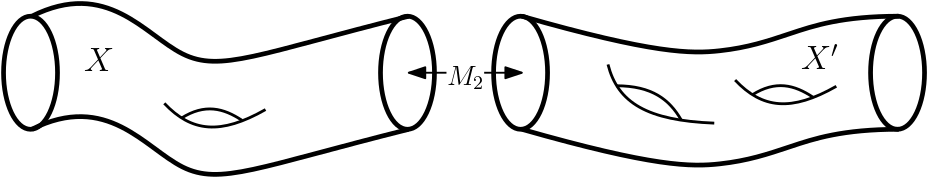}
\caption{Gluing two nodefolds along a common boundary.}
\label{img:gluing}
\end{figure}

\begin{prop}\label{lem:gluing} Suppose that $A \subseteq X$ and $A' \subseteq X'$ are finite subsets with $A \cap M_2 = A' \cap M_2$. Moreover, suppose that $A \cap M_2$ meets all the connected components of $M_2$. Then, the
following commutative diagram of groupoids induced by the natural inclusions
\[
\begin{displaystyle}
   \xymatrix
   {
    \Pi(X \cup_{M_2} X', A \cup A') &  \Pi(X', A') \ar[l] \\
   \Pi(X, A) \ar[u] & \Pi(M_2, A \cap M_2) \ar[u] \ar[l]\\
  	}
\end{displaystyle}
\]
is a pushout.
\end{prop}

\begin{proof}
The proof is essentially Seifert-van Kampen theorem for fundamental groupoids, as proven in \cite{Brown:1967}. Take open collarings $U \subseteq X$ and $U' \subseteq X'$ around $M_2$ i.e.\ such that $U, U' \cong M_2 \times [0,1)$. Shrinking if necessary, we can also suppose that $A \cap U = A \cap M_2$ and $A' \cap U' = A' \cap M_2$.

Then, using Seifert-van Kampen theorem with the open sets $X \cup_{M_2} U'$ and $X' \cup_{M_2} U$ of $X \cup_{M_2} X'$, we get that the following diagram is a pushout
\[
\begin{displaystyle}
   \xymatrix
   {
    \Pi(X \cup_{M_2} X', A \cup A') &  \Pi(X' \cup_{M_2} U, A') \ar[l] \\
   \Pi(X \cup_{M_2} U', A) \ar[u] & \Pi(U \cup_{M_2} U', A \cap M_2) \ar[u] \ar[l]\\
  	}
\end{displaystyle}
\]
Now, the result follows using that $(X \cup_{M_2} U', A)$ has the homotopy type of $(X, A)$, $(X' \cup_{M_2} U, A')$ has the homotopy type of $(X', A')$, and $(U \cup_{M_2} U', A \cap M_2)$ has the homotopy type of $(M_2, A \cap M_2)$.
\end{proof}

Furthermore, since the functor $\Hom(-, G)$ is co-continuous, we can extend this result to representation varieties.

\begin{cor}
Under the hypotheses of Proposition \ref{lem:gluing}, for any algebraic group $G$, we have that the representation variety of the gluing is the fibered product
$$
	\Rep{G}(X \cup_{M_2} X', A \cap A') = \Rep{G}(X, A) \times_{\Rep{G}(M_2, A \cap M_2)} \Rep{G}(X', A').
$$
\end{cor}

These results can be easily extended to the parabolic setting. Suppose that we fix an algebraic group $G$ and that $Q$ and $Q'$ are parabolic structures on two nodefolds $(X, A)$ and $(X', A')$. Suppose also that we want to glue $X$ and $X'$ along a common boundary $M_2$ as above. This gluing can be extended to the parabolic structure provided that the induced parabolic structures agree $Q|_{M_2} = Q'|_{M_2}$, giving rise to a new parabolic structure on $X \cup_{M_2} X'$, denoted $Q \cup_{M_2} Q'$. Roughly speaking, it is given by gluing together the submanifolds $S_i$, $S_j'$ along the boundary $M_2$, and labelling the result with $\lambda_i = \lambda_j'$.

At the level of representation varieties, observe that this construction is compatible with the description of $\Rep{G}(X \cup_{M_2} X', Q \cup_{M_2} Q')$ as a product variety, in the spirit of Equation (\ref{eq:rep-variety-general}). In this way, we obtain also the following result.

\begin{cor}\label{cor:gluing-parabolic}
In the aforementioned hypothesis, for any algebraic group $G$, the parabolic $G$-representation variety of the gluing is the fibered product
$$
	\Rep{G}(X \cup_{M_2} X', A \cup A', Q \cup_{M_2} Q') = \Rep{G}(X, A, Q) \times_{\Rep{G}(M_2, A \cap M_2, Q|_{M_2})} \Rep{G}(X', A', Q').
$$
\end{cor}

\section{TQFTs for representation varieties over nodefolds}\label{sec:tqft}

In this section, we shall construct a lax monoidal Topological Quantum Field Theory (TQFT) computing the virtual class in the Grothendieck ring of algebraic varieties, of representation varieties over nodefolds. This construction extends the TQFT built in \cite{GPLM-2017} for topological manifolds.

\subsection{Grothendieck ring of algebraic varieties}\label{sec:grothendieck-ring}

From now on, we will work on a fixed algebraically closed field $k$. Let $\Var{k}$ be the category of algebraic varieties over $k$, that is of integral separated schemes of finite type over $k$, with regular morphisms between them. Together with the disjoint union of varieties and their cartesian product (more preciselly, their fibered product over $k$), the isomorphism classes of objects of $\Var{k}$ form a semi-ring.

From it we can construct the so-called \emph{Grothendieck ring} of algebraic varieties, or $K$-theory ring, denoted $\K{\Var{k}}$. Explicitly, $\K{\Var{k}}$ is the ring generated by the isomorphism classes of algebraic varieties, denoted $[Z]$ for $Z \in \Var{k}$ and usually referred to as the \emph{virtual class} of $Z$, subject to the relations
$$
	[Z_1 \sqcup Z_2] = [Z_1] + [Z_2], \quad [Z_1 \times Z_2] = [Z_1] \cdot [Z_2].
$$
Indeed, combining both relations, we can get a slightly more general multiplicativity property. 

\begin{prop}
Suppose that $\pi: Z \to B$ is a locally trivial fibration of algebraic varieties in the Zariski topology with fiber $F$. Then, we have
$
	[Z] = [B][F].
$
\end{prop}

\begin{proof}
Choose an open cover $\left\{U_i\right\}$ of $B$ such that $\pi|_{U_i}$ is trivial. Then, by standard multiplicativity, we have that $[\pi^{-1}(U_i)] = [U_i \times F] = [U_i][F]$. In this way, coming back to the global situation we have
$$
	[Z] = \left[\bigsqcup_{i} \pi^{-1}(U_i)\right] = \sum_{i} \left[\pi^{-1}(U_i)\right] = \sum_{i} \left[U_i\right] [F] = [F] \sum_{i} \left[U_i\right] = [F][B].
$$
\end{proof}

An important element in $\K{\Var{k}}$ is the class of the affine line, typically denoted $q = [k]$ and called the \textit{Lefschetz motive}. The notation $\mathbb{L} = [k]$ is also very standard, specially in motivic theory, but we will avoid it in this paper.

\begin{rmk}\label{rmk:Kvar-integral}
Despite its simple definition, very little is known about the ring structure of $\K{\Var{k}}$, even in the complex case. For almost fifty years, it was conjectured that it is an integral domain. Finally, the answer is no but, more strikingly, the Lefschetz motive is a zero divisor \cite{borisov2018class}. Note that this captures very subtle properties of algebraic varieties, namely that there exists non-isomorphic algebraic varieties $Z_1, Z_2 \in \K{\Var{k}}$ such that $Z_1 \times k$ and $Z_2 \times k$ are isomorphic.

Through this paper, we will need to divide by $q$, $q+1$ and $q-1$ several times. For this reason, instead of working in $\K{\Var{k}}$, we will work on the localization of this ring by the multiplicative set generated by $q$, $q+1$ and $q-1$. The requirement of this localization will be clear from the context so, to lighten notation, we will denote this localization also by $\K{\Var{k}}$. Observe that this localization has the effect that, all the zero divisor partners of $q$, $q+1$ and $q-1$ are annihilated.
\end{rmk}

In the complex case $k = \CC$, $\K{\Var{\CC}}$ extends some Hodge-theoretic invariants called the $E$-polynomial (a.k.a. Deligne-Hodge polynomial). Given a complex algebraic variety $Z$, its rational compactly supported cohomology, $H_c^k(Z, \QQ)$, is endowed with a natural mixed Hodge structure given by two filtrations: an increasing weigh filtration $W_\bullet$ and a decreasing Hodge filtration $F^\bullet$. From them, we can consider the so-called \emph{Hodge numbers} as
$$
	h^{k;p,q}_c(Z) = \dim \mathrm{Gr}_F^{p} \left(\left(\mathrm{Gr}^W_{p+q} H_c^k(Z, \QQ)\right) \otimes_\QQ \CC\right),
$$
where $\mathrm{Gr}$ denoted the graded complex of a filtration. These numbers can be collected in the \emph{$E$-polynomial}
$$
	e(Z)(u,v) = \sum_{k} (-1)^kh^{k;p,q}_c(Z)\,u^{p}v^q,
$$
which is a polynomial in $\ZZ[u,v]$. By additivity of alternating sums of complexes and K\"unneth isomorphism, we get that the $E$-polynomial defines a semi-ring homomorphism $e: \Var{k} \to \ZZ[u,v]$ so it can be extended to a ring homomorphism $e: \K{\Var{k}} \to \ZZ[u,v]$.

In particular, under this morphism we have that $e(q) = e([k]) = uv$. Moreover, if $[Z] \in \K{\Var{k}}$ lies in the subring generated by the Lefschetz motive, that is, $[Z] = P(q)$ for some polynomial $P \in \ZZ[q]$, then the previous computation shows that $e([Z]) = P(uv)$. In this vein, for algebraic varieties whose virtual class is generated by $q$, the virtual class extends the $E$-polynomial.

\begin{rmk}
If $[Z]$ is generated by $q$, then $e(Z)$ is a polynomial in the variable $uv$, which means that $h^{k;p,q}_c(Z) = 0$ whenever $p \neq q$. In general, the varieties with such vanishing of Hodge numbers are called \emph{balanced}. In this situation, it is customary to write the $E$-polynomial in the variable $q = uv$, which justifies our notation for the Lefschetz motive.
\end{rmk}

The aim of this paper is to construct a TQFT computing the virtual classes of $G$-representation varieties over nodefolds. Moreover, we will focus on the case $G = \SL{2}(k)$ where we shall perform all the calculations explicitly. For this reason, it is important to understand the virtual classes of some related groups.
\begin{itemize}
	\item $[\GL{2}(k)] = q^4 - q^3 - q^2 + q$. This follows by noticing that we have a locally trivial fibration in the Zariski topology $\GL{2}(k) \to k^2 - \left\{(0,0)\right\}$ given by $A \mapsto Ae_1$ where $e_1$ is the first vector of the canonical basis of $k^2$. The fiber of this map is the set of vectors of $k^2$ that do not lie in the line spanned by $Ae_1$. In this way
	$$
		[\GL{2}(k)] = \left[k^2 - \left\{(0,0)\right\}\right]\left[k^2 - k\right] = (q^2-1)(q^2-q).
	$$
	Similar expressions for the virtual class of $[\GL{n}(k)]$ for $n \geq 2$ can be obtained with a similar argument.
	\item $[\PGL{2}(k)] = q^3-q$. Consider the quotient map $\GL{n}(k) \to \PGL{n}(k) = \GL{n}(k)/k^*$, with $k^* = k-\left\{0\right\}$. This is a locally trivial fibration in the Zariski topology with fiber $k^*$, so $[\PGL{n}(k)] = [\GL{n}(k)]/(q-1)$. Using the formula above for $[\GL{2}(k)]$ the result follows.
	\item $[\SL{2}(k)] = q^3-q$. The argument is similar to the one of $\PGL{2}(k)$ but, now, instead of considering the quotient map, we consider the map $\GL{n}(k) \to \SL{n}(k)$ that, to a matrix $A \in \GL{n}(k)$ assignes the same matrix but with the first column divided by $\det(A) \neq 0$. Again, this is a locally trivial map in the Zariski topology with fiber $k^*$.
\end{itemize}

\subsection{Functoriality of Grothendieck rings}

The previous construction of the Grothendieck ring of algebraic varieties can be performed relatively to a base variety. In this scenario, new functoriality properties arise that will be very useful for our construction of a TQFT. 

Given an algebraic variety $S$, let us denote by $\Varrel{S}$ the category of algebraic varieties over $S$. Explicitly, this category has, as objects regular morphisms $Z \to S$ and its morphisms are regular maps $Z \to Z'$ preserving the base maps. In particular, if $S = \star$ is the singleton variety then $\Varrel{\star} = \Var{k}$ is the usual category of algebraic varieties.

Again, together with the disjoint union 
$\sqcup$ of algebraic varieties, and the fibered product $\times_S$ over $S$, 
we may consider its associated Grothendieck ring $\K{\Varrel{S}}$. The element of $\K{\Varrel{S}}$ induced by a morphism $h: Z \to S$ will be denoted as $[(Z, h)]_S \in \K{\Varrel{S}}$, or just $[Z]_S$ when the map is clear from the context. In this notation, the unit of $\K{\Varrel{S}}$ is $\Unit{S} = [S, \Id_S]_S$.

This construction exhibits some important functoriality properties that will be useful for our construction. Suppose that $f: S_1 \to S_2$ is a regular morphism. It induces:
\begin{itemize}
	\item A ring homomorphism $f^*:\K{\Varrel{S_2}} \to \K{\Varrel{S_1}}$ given by $f^*[Z]_{S_2} = [Z \times_{S_2} S_1]_{S_1}$. In particular, taking the projection map $i: S \to \star$ we get a ring homomorphism $i^*: \K{\Var{k}} \to \K{\Varrel{S}}$ that endows the rings $\K{\Varrel{S}}$ with a natural structure of $\K{\Var{k}}$-module that corresponds to the cartesian product.
	\item A $\K{\Var{k}}$-module homomorphism $f_!: \K{\Varrel{S_1}} \to \K{\Varrel{S_2}}$ given by $f_![(Z, h)]_{S_1} = [(Z, f \circ h)]_{S_2}$. In general $f_!$ is not a ring homomorphism but, for $[Z_1] \in \K{\Varrel{S_1}}$ and $[Z_2] \in \K{\Varrel{S_2}}$, the projection formula $f_!([Z_2] \times_{S_2} f^*[Z_1]) = f_![Z_2] \times_{S_1} [Z_1]$ holds, which implies that $f_!$ is a $\K{\Var{k}}$-module homomorphism.
\end{itemize} 

The induced morphisms are functorial, in the sense that $(g \circ f)^* = f^* \circ g^*$ and $(g \circ f)_! = g_! \circ f_!$. In particular, if $i: T \hookrightarrow S$ is an inclusion, then $i^*f^* = f|_{T}^*$.

\subsection{Construction of the TQFT}

The notion of TQFT is very useful for computing algebraic invariants attached to manifolds. Recall from \cite{Atiyah:1988} that a $n$-dimensional Topological Quantum Field Theory (TQFT) is a monoidal symmetric functor
$$
	Z: \Bord{n} \to \Mod{R}.
$$
Here, $\Bord{n}$ is the category of $n$-dimensional bordisms, that is, the category whose objects are closed manifolds of dimension $n-1$ and a morphism $X: M_1 \to M_2$ is a bordism class between $M_1$ and $M_2$ i.e.\ a compact $n$-dimensional manifold with $\partial X = M_1 \sqcup M_2$ up to boundary-preserving homeomorphism. It becomes a monoidal category with disjoint union (of manifolds and morphisms). The target category, $\Mod{R}$, is the usual category of $R$-modules and homomorphisms, with monoidal structure given by tensor product. There exists in the literature many constructions of TQFT that had provided deep insight in different areas, like  \cite{Carlsson-Rodriguez-Villegas,Diaconescu:2017,Hausel-Letellier-Villegas:2013,Mozgovoy:2012}.

\begin{rmk}
Due to the physical inception of many of these TQFTs, it is customary to impose that the objects of $\Bord{n}$ are smooth oriented manifolds and the bordisms between them are also smooth and oriented, in such a way that the orientation agrees with the ones at the boundaries. However, we will not this these extra structures in this paper.
\end{rmk}

In \cite{GP-2018} (see also \cite{GPLM-2017}), it was constructed a sort of TQFT to compute virtual classes of representation varieties over closed manifolds. Here, we will extend such construction to nodefolds. For that purpose, first we need to enlarge the category $\Bord{n}$ to also include nodefolds, as well as some extra pieces of data needed for defining representation varieties.

From now on, we fix an algebraic group $G$ and a collection $\Lambda$ of subvarieties of $G$ invariants under conjugation (typically, $\Lambda$ is a collection of conjugacy classes).

\begin{defn}
Let $n \geq 1$. We define the \emph{category of $n$-dimensional pairs of nodefold bordisms with parabolic data $\Lambda$}, $\NBord{n}{\Lambda}$, as the $2$-category given by:
\begin{itemize}
	\item Objects: The objects are triples $(M, A, Q)$, where $M$ is a closed manifold of dimension $n-1$, $A \subseteq M$ is a finite set of points meeting each connected component of $M$ and $Q$ is a $G$-parabolic structure on $(M, A)$. The empty manifold, $\emptyset$, is also an object.
	\item Morphisms: A morphism $(M_1, A_1, Q_1) \to (M_2, A_2, Q_2)$ is a triple $(X, A, Q)$ where $X$ is a $n$-dimensional nodefold with $\partial X = M_1 \sqcup M_2$, $A \subseteq X$ is a finite set of points meeting each connected component of $X$ and with $A \cap M_1 = A_1$ and $A \cap M_2 = A_2$, and $Q$ is a parabolic structure on $X$ with $Q|_{M_1} = Q_1$ and $Q|_{M_2} = Q_2$. Morphisms are defined up to homotopy equivalence, that is, $(X, A, Q)$ and $(X', A', Q')$ are declared to be equal if there exists homotopy equivalences $f: (X,A) \to (X',A')$ and $g: (X', A') \to (X, A)$ such that $f^*Q' = Q$, $g^*Q = Q'$ and the homotopies $g \circ f \sim \Id_{X}$ and $f \circ g \sim \Id_{X'}$ also preserve $Q$ and $Q'$.
	
	 In addition, we allow an object $(M, A, Q)$ to be seen as a morphism $(M,A,Q): (M,A,Q) \to (M,A,Q)$ (as a very thin bordism), and we decrete it as the identity morphism. Composition of morphisms is given by gluing.
	 \item $2$-Morphisms: A $2$-morphism between bordisms $(X, A, Q)$ and $(X', A', Q')$ is given by a conic degeneration $f: X \to X'$ with $Q \subseteq f^*Q'$. Vertical composition is given by composition of the degenerations. 
\end{itemize}
\end{defn}
 
With this definition, $\NBord{n}{\Lambda}$ also has a natural monoidal structure given by disjoint union.

Another important category is the category of algebras with twists. Let $R$ be a fixed commutative and unitary ring. Let $M$ and $N$ be two $R$-algebras and suppose that $f,g: M \to N$ are two homomorphisms as $R$-modules. We say that $g$ is an \emph{immediate twist} of $f$ if there exists $R$-algebras $L$ and $L'$, an algebra homomorphism $\alpha^{alg}: L \to L'$ and a $R$-module homomorphism $\alpha^{mod}: L' \to L$ such that the following diagram commutes.
\[
\begin{displaystyle}
   \xymatrix
   { && L \ar[rrd]^{f^{mod}} \ar@/_0.8pc/[dd]_{\alpha^{alg}}  && \\
   M \ar[rru]^{f^{alg}}\ar[rrd]_{g^{alg}} &&  && N \\
   && L'\ar[rru]_{g^{mod}} \ar@/_0.8pc/[uu]_{\alpha^{mod}} &&
   }
\end{displaystyle}   
\]
Here, $f^{alg}$ and $g^{alg}$ are $R$-algebra homomorphisms, $f^{mod}$ and $g^{mod}$ are $R$-modules homomorphisms and $f = f^{mod} \circ f^{alg}$ and $g = g^{mod} \circ g^{alg}$.

In general, given $f,g: M \to N$ two $R$-module homomorphisms, we say that $g$ is a \emph{twist} of $f$ if there exists a finite sequence $f=f_0, f_1, \ldots, f_r = g: M \to M$ of homomorphisms such that $f_{i+1}$ is an immediate twist of $f_i$ for $0 \leq i \leq r-1$.


In that case, we define the \emph{category of $R$-algebras with twists}, $\Modt{R}$, as the category whose objects are $R$-algebras, its $1$-morphisms are $R$-modules homomorphisms and, given homomorphisms $f$ and $g$, a $2$-morphism $f \Rightarrow g$ is a twist from $f$ to $g$. Composition of $2$-cells is juxtaposition of twists. With this definition, $\Modt{R}$ has a $2$-category structure. Moreover it is a monoidal category with the usual tensor product. 

\begin{defn}
Fix a commutative ring with unit, $R$. A lax monoidal $2$-functor
$$
	Z: \NBord{n}{\Lambda} \to \Mod{R}_t
$$
is called a conic lax monoidal TQFT.
\end{defn}

\begin{rmk}
Recall that lax monoidality means that $Z$ preserves the unit of the monoidal structure, $Z(\emptyset) = R$, but we only have a morphism
$$
	Z(M, A, Q) \otimes_R Z(M', A', Q') \to Z\left((M, A, Q) \sqcup (M', A', Q')\right)
$$
which is no longer an isomorphism.
\end{rmk}

The aim of this section is to construct a conic lax monoidal TQFT computing virtual classes of representation varieties. It will be constructed as a composition of two $2$-functors
$$
	\NBord{n}{\Lambda} \stackrel{\cF_G}{\longrightarrow} \Span{\Var{k}} \stackrel{\cQ}{\longrightarrow} \Mod{\K{\Var{k}}}.
$$

\begin{rmk}
Recall that $\Span{\Var{k}}$ is the $2$-category of spans of algebraic varieties. More precisely, the objects of this category are algebraic varieties and a morphism between algebraic varieties $Z$ and $Z'$ is a span of regular maps of the form

$$
\xymatrix{
&\ar[dl]_{f} S \ar[dr]^{g}&\\
Z & & Z'
}
$$
Given two spans $(S_1, f_1, g_1): Z \to Z'$ and $(S_2, f_2, g_2): Z' \to Z''$ its composition is given by pullback. Explicitly, we define $(Z_2, f_2, g_2) \circ (Z_1, f_1,g_1) = (S_1 \times_{Z'} S_2, f_1 \circ f_2', g_2 \circ g_1')$, where $f_2', g_1'$ are the morphisms in the pullback diagram
$$
\xymatrix{
&&\ar[dl]_{f'_2}S_1 \times_{Z'} S_2\ar[dr]^{g'_1}&&\\
&\ar[dl]_{f_1}S_1\ar[dr]^{g_1}&&\ar[dl]_{f_2}S_2\ar[dr]^{g_2}&\\
Z&&Z'&&Z''
}
$$

A $2$-morphisms of $\Span{\Var{k}}$ between spans $(f,g,S), (f',g',S'): Z_1 \to Z_2$ is given by a regular morphism $\alpha: S' \to S$ such that the following diagram commutes

\[
\begin{displaystyle}
   \xymatrix
   {	& S' \ar[rd]^{g'} \ar[ld]_{f'} \ar[dd]^\alpha & \\
   		Z &  & Z_2\\
	&S  \ar[ru]_{g} \ar[lu]^{f}& 
   }
\end{displaystyle}   
\]

Finally, $\Span{\Var{k}}$ also exhibits a natural monoidal structure given by usual cartesian product of varieties and regular morphisms.
\end{rmk}

The functor $\cQ: \Span{\Var{k}} \to \Mod{\K{\Var{k}}}_t$ is called the \emph{quantization} of the TQFT. It is a lax monoidal functor that was constructed in \cite{GP-2018}. Roughly speaking, it is the motivic version of a Fourier-Mukai transform with identity kernel that assigns:
\begin{itemize}
	\item Objects: Given an algebraic variety $X \in \Span{\Var{k}}$, it assigns $\cQ(X) = \K{\Varrel{X}}$, the Grothendieck ring of the category of algebraic varieties over $X$, seen as a $\K{\Var{k}}$-module.
	\item Morphisms: Given a span $(f,g,S): X \to Y$, it associates the $\K{\Var{k}}$-module homomorphism $g_! \circ f^*: \K{\Varrel{X}} \to \K{\Varrel{Y}}$.
	\item $2$-Morphisms: To a $2$-morphism $\alpha: (f,g,S) \Rightarrow (f',g',S')$, it assigns the twisting immediate twisting $(\alpha^* = \alpha^{alg}, \alpha_! = \alpha^{mod}): g_! \circ f^* \Rightarrow g'_! \circ f'^*$.
\end{itemize}

With respect to the functor $\cF_G: \NBord{n}{\Lambda} \to \Span{\Var{k}}$, usually referred to as the \emph{field theory}, the construction works as follows:
\begin{itemize}
	\item Objects: To an object $(M, A, Q)$, it assigns the associated parabolic representation variety $\Rep{G}(M, A, Q)$.
	\item Morphisms: Given a nodefold bordism $(X, A, Q): (M_1, A_1, Q_1) \to (M_2, A_2, Q_2)$, let us denote by $j_1: (M_1, A_1, Q_1) \hookrightarrow (X, A, Q)$ and $j_2: (M_2, A_2, Q_2) \hookrightarrow (X, A, Q)$ the inclusion maps as boundaries. Then we associate to this the span
	
$$
\xymatrix{
&\ar[dl]_{i_2} \Rep{G}(W, A, Q) \ar[dr]^{i_2}&\\
\Rep{G}(M_1, A_1, Q_1) && \Rep{G}(M_2, A_2, Q_2)
}
$$
where $i_1, i_2$ are the maps induce by the inclusions $j_1, j_2$ at the level of representation varietis.

	\item $2$-Morphisms: To a $2$-morphism $(X, A, Q) \to (X', A', Q')$ given by a conic degeneration $f: X \to X'$, we associate the regular map $\Rep{G}(f): \Rep{G}(X', A', Q') \to \Rep{G}(X, A, Q)$. Observe that, by construction, this map interwines with the inclusion morphisms.
\end{itemize}

By its very definition, the assignment $\cF_G$ commutes with vertical composition. Moreover, by Corollary \ref{cor:gluing-parabolic}, $\cF_G$ also commutes with horizontal composition. Finally, since $\Rep{G}((X, A, Q) \sqcup (X', A',Q')) = \Rep{G}((X, A, Q) \times \Rep{G}(X', A',Q'))$, the functor $\cF_G$ is monoidal.

Set $Z_G = \cQ \circ \cF_G: \NBord{n}{\Lambda} \to \Mod{\K{\Var{k}}}$, which is a lax monoidal $2$-functor. Let $(X, A, Q): \emptyset \to \emptyset$, which is given by an $n$-dimensional compact nodefold $X$ without boundary. Observe that $\Rep{G}(\emptyset) = \star$ is the singleton variety, so denoting by $c: \Rep{G}(X, A, Q) \to \star$ the projection map, we have that
$$
	Z_G(X, A, Q) = \cQ \circ \cF_G(X, A, Q) = \cQ\left(\star \stackrel{c}{\leftarrow} \Rep{G}(X, A, Q) \stackrel{c}{\rightarrow} \Rep{G}(\star) = 1\right) = c_!c^*.
$$
In particular, if we apply this map to the unit $[\star]_\star \in \K{\Var{k}}$, and using that $c^*$ is a ring homomorphism, we have that
$$
	Z_G(X, A, Q)([\star]_\star) = c_!c^* 1_* = c_![\Rep{G}(X, A, Q)]_{\Rep{G}(X, A, Q)} = [\Rep{G}(X, A, Q)]_*.
$$

This is nothing but the virtual class of the associated parabolic representation variety in $\K{\Var{k}}$. Therefore, putting all together we have proven the following result.

\begin{thm}
For any algebraic group $G$ and any $n \geq 1$, there exists a conic lax monoidal TQFT
$$
	Z_G: \NBord{n}{\Lambda} \to \Mod{\K{\Var{k}}}
$$
computing the virtual classes of parabolic representation varieties over nodefolds.
\end{thm}

\subsection{Effect of conic degenerations}\label{sec:conic-deg-tqft}

In order to get some flavour about the behavior of $\Zs{G}$ with respect to degenerations, suppose that $(W, A, Q), (W',A', Q'): (M_1,A_1, Q_1) \to (M_2, A_2, Q_2)$ are two nodefold bordisms and that $f: (W, A, Q) \to (W',A', Q')$ is a degeneration. Hence, under the field theory, we have a commutative diagram of spans
\[
\begin{displaystyle}
   \xymatrix
   {
   & \Rep{G}(W, A, Q) \ar[rd]^{i_2}\ar[ld]_{i_1} & \\
   	\Rep{G}(M_1,A_1, Q_1) & & \Rep{G}(M_2,A_2, Q_2) \\
   	& \Rep{G}(W', A', Q') \ar[uu]_{f} \ar[ru]_{i_2'} \ar[lu]^{i_1'} &
  	}
\end{displaystyle}
\]
where we have overloaded the notation $\Rep{G}(f) = f: \Rep{G}(W', A', Q') \to \Rep{G}(W, A, Q)$. Therefore, we have that $\Zs{G}(W,A, Q) = (i_2)_! (i_1)^*$ and $\Zs{G}(W',A', Q') = (i_2)_! f_!f^*(i_1)^*$. In this way, the endomorphism $f_!f^*: \K{\RVar{k}{\Rep{G}(W, A, Q)}} \to \K{\RVar{k}{\Rep{G}(W, A, Q')}}$ gives the corresponding immediate twist $\Zs{G}(W, A, Q) \Rightarrow \Zs{G}(W', A', Q')$.

A particular example of this phenomenon appears when $W$ is a honest compact smooth manifold with a degeneration to a nodefold $f:W \to W'$, a situation that is called a \emph{normalization} of the nodefold $W'$. 

\begin{lem}
Let $f: W \to W'$ be a normalization with $W$ of dimension $n$. Then, around any conic point $p' \in W'$, $W'$ is locally homeomorphic to $\Cone{\bigsqcup_rS^{n-1} }$. The number $r > 1$ is called the number of \emph{branches} of $W'$ at $p'$. 
\begin{proof}
Pick $p \in f^{-1}(p')$. Since $W$ is a manifold, there exists an open set $U \subseteq W$, with $U \cap f^{-1}(\conic{W'})=\left\{p\right\}$, homeomorphic to the $n$-dimensional open ball $B^n = \Cone{S^{n-1}}$. Hence, since $f$ is a degeneration, $f: U-\left\{p\right\} \to f(U)-\left\{p'\right\}$ is a homeomorphism so $f|_U: U \to f(U)$ is a continuous bijective map. Thus, maybe restricting $U$, we find that $f|_U$ is an homeomorphism.

Therefore, locally around $p'$, $W'$ is homeomorphic to the gluing of $r=|f^{-1}(p')|$ open balls along a common interior point, which is precisely $\Cone{S^{n-1} \sqcup \ldots \sqcup S^{n-1}}$.
\end{proof}
\end{lem}

Now, let $f: W \to W'$ be a normalization with $W$ a closed connected $n$-dimensional manifold. For simplicity, we will suppose that $W'$ as a unique conic point $p' \in W'$ and let $f^{-1}(p')=\left\{p_1, \ldots, p_r\right\}$. By the previous proposition, there exist open neighbourhoods $U_i \subseteq W$ of $p_i$ homeomorphic to open balls and an open set $U' \subseteq W'$ of $p'$ such that $f$ gives an homeomorphism of $U'$ with the gluing of $U_1 \cup \ldots \cup U_r$ along their center.

In this way, setting $\tilde{W} = W - U_1 - \ldots - U_r$ and $\tilde{W'} = W' - U'$ we obtain decompositions of the bordisms $W, W': \emptyset \to \emptyset$ as
$$
	W = \left(\bigsqcup_r D^n\right) \circ \tilde{W}, \hspace{1cm} W' = \Cone{\bigsqcup_r S^{n-1}} \circ \tilde{W}'.
$$
Here $D^n$ denotes the closed $n$-dimensional ball and we are seeing $\bigsqcup_r D^n$ and $\Cone{\bigsqcup_r S^{n-1}}$ both as bordisms $ \bigsqcup_r S^{n-1} \to \emptyset$.

Now, let us choose $A' \subseteq W'$ with $p' \not\in A'$ such that it contains a single point on each of the $r$ copies of the boundaries $S^{n-1} \subseteq W'$ and does not meet the interior of the cone. Let us also denote $A = f^{-1}(A') \subseteq W$. Observe that we have $\Rep{G}\left(\bigsqcup_r D^n, A\right) = \prod_r\Rep{G}\left(\star\right) = 1$ and $\Rep{G}\left(\Cone{\bigsqcup_r S^{n-1}}, A'\right) = \Rep{G}\left(\Cone{\bigsqcup_r S^{n-1}}\right) \times G^{r-1} = \Rep{G}\left(\star\right) \times G^{r-1} = G^{r-1}$. Therefore, under the field theory $\cF_G$, we obtain the commutative diagram of spans
\[
\begin{displaystyle}
   \xymatrix
   {
   & 1 \ar[rrd] \ar[ld]_i && \\
   	\prod_r \Rep{G}\left(S^{n-1}, \star\right) & && 1 = \Rep{G}(\emptyset)\\
   	& G^{r-1} \ar[uu]_p \ar[rru]_{p} \ar[lu] &&
  	}
\end{displaystyle}
\]
Here, the uppermost span is the corresponding one to $(\bigsqcup_r D^n, A): \bigsqcup_r (S^{n-1}, \star) \to \emptyset$ and the left arrow is the inclusion of $1 = (1, \ldots, 1)$ into $\prod_r \Rep{G}\left(S^{n-1}, \star\right)$. The lowermost span is the corresponding to $(\Cone{\sqcup_r S^{n-1}}, A'): \bigsqcup_r (S^{n-1}, \star) \to \emptyset$ and the right and vertical arrows are the projection $p: G^{r-1} \to 1$.

\begin{rmk}
If $n > 2$, then $S^{n-1}$ is simply connected and, thus, the map $i$ is the identity. On the other hand, if $n = 2$, then $\Rep{G}\left(S^{n-1}, \star\right) = G$ so $\prod_r \Rep{G}\left(S^{n-1}, \star\right) = G^{r}$. However, even in that case, the left lowermost map $G^{r-1} \to \prod_r \Rep{G}\left(S^{n-1}, \star\right) = G^r$ is not an inclusion, but the projection onto $1 \in G^r$.
\end{rmk}

Therefore, we obtain that
$$
	\Zs{G}\left(\sqcup_r D^{n}, A\right) = i^*, \hspace{0.5cm} \Zs{G}\left(\Cone{\sqcup_r S^{n-1}}, A\right) = p_!p^*i^* = [G]^{r-1} \times i^*.
$$
Hence, we get that
\begin{align*}
	[\Rep{G}(W')] &\times [G]^{|A|-1} = \Zs{G}(W', A)(1) = \Zs{G}(\Cone{\sqcup_r S^{n-1}}, A) \circ \Zs{G}(\tilde{W}, A)(1) \\
	&=  i^*\Zs{G}(\tilde{W}, A)(1) \times [G]^{r-1} = \Zs{G}(W, A)(1) \times [G]^{r-1} = [\Rep{G}(W)] \times [G]^{r+ |A|-2}.
\end{align*}
Therefore, if we work the localization of $\K{\Var{k}}$ by $[G]$, we get that $[\Rep{G}(W')] = [\Rep{G}(W)] \times [G]^{r-1}$.


Procceding analogously in the case of $l>1$ conic points in $W'$ with $r_1, \ldots, r_l$ branches, we finally obtain that
\begin{align}\label{eq:rep-conic}
	[\Rep{G}(W')] = [\Rep{G}(W)] \times [G]^{r_1 + \ldots r_l-1}.
\end{align}
Analogous results can be obtained if $W$ and $W'$ are not closed or in the case of parabolic structures.


A very important example of nodefold comes from plane curves, as the following result shows.

\begin{prop}\label{prop:complex-plane-conic}
Let $X$ be a complex projective irreducible plane curve, maybe with singularities. Then $X$ is a nodefold whose conic points $p \in X$ are those singular points with $r > 1$ branches. In that case, $p$ has an open neighbourhood homeomorphic to $\Cone{S^1 \sqcup \stackrel{(r)}{\ldots} \sqcup S^1}$.
\begin{proof}
If $p \in X$ is a regular point, then $X$ around $p$ is a topological manifold. Hence, we can suppose that $p \in X$ is a singular point with $r \geq 1$ branches. Let $B_\epsilon(p) \subseteq \CC^2$ be a small open ball of radius $\epsilon > 0$ with center at $p$ and let $S_\epsilon(p)$ be its boundary sphere. By \cite[Corollary 2.9]{milnor2016singular}, for $\epsilon$ small enough, $S_\epsilon(p) \cap X$ is a (smooth) link in $S_\epsilon(p) \cong S^3$ so it is homeomorphic to $B = S^1 \sqcup \stackrel{(r)}{\ldots} \sqcup S^1$, where $r$ is the number of branches of  $X$ around $p$. Moreover, by \cite[Theorem 2.10]{milnor2016singular}, $B_\epsilon(p) \cap X$ is homeomorphic to $\Cone{B}$, proving that $X$ is locally conic at $p$ of the claimed form. Finally, observe that if $r=1$, $\Cone{S^1}$ is homeomorphic to an open ball, so $p$ is smooth.
\end{proof}
\end{prop}

\begin{rmk}
Indeed, Theorem 2.10 of \cite{milnor2016singular} holds for any irreducible hypersurface. Hence, in general, we actually have that any irreducible hypersurface is a nodefold. However, for real dimension of the hypersurface $n \geq 4$, there is a variety of compact manifolds of dimension $n - 1$ so, for higher dimensions, it is not possible to identify the local structure of a conic point so easily.
\end{rmk}

\begin{ex}
If $p \in X$ is a nodal point with $r>1$ branches, then it has an open neighbourhood homemorphic to the cone over $S^1 \sqcup \stackrel{(r)}{\ldots} \sqcup S^1$. A cusp point is a smooth point (in the nodefold sense) since it has $r=1$ branches.
\end{ex}

The previous result fully classifies complex projective plane curves topologically. First observe that, for each $r > 0$, there exists a smooth compact surface $W_r$ with boundary $\partial W_r = \bigsqcup_r S^1$ that gives rise to a degeneration $\varpi_r: W_r \to \Cone{\bigsqcup_r S^1}$ ramified at the vertex of the cone. $W_r$ is obtained by ``blowing-up'' the cone $\Cone{\bigsqcup_rS^1}$ at it vertex, as shown in Figure \ref{img:degeneration}. 

\begin{figure}[h]
\includegraphics[width=10cm]{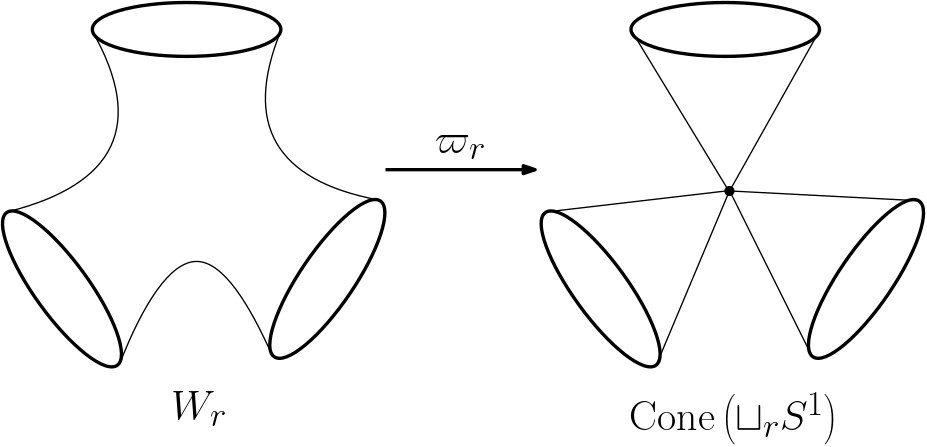}
\caption{Normalization of a conic point.}
\label{img:degeneration}
\end{figure}

Now, let $X$ be a complex projective plane curve that, according to Proposition \ref{prop:complex-plane-conic}, topologically it is a node-surface with $p_1, \ldots, p_l$ conic points with $r_1, \ldots, r_l$ branches each. By removing small neighborhoods of $X$ around the conic points $p_i$ and replacing them by $W_{r_i}$, we obtain a smooth closed surface $\Sigma$ and a normalization $\varpi: \Sigma \to X$. Hence, using equation (\ref{eq:rep-conic}) we get that
\begin{equation}\label{eq:rep-conic-surface}
	[\Rep{G}(X)] = [\Rep{G}(\Sigma)] \times [G]^{r_1 + \ldots r_l-1}.
\end{equation}
This shows that $\Sigma$ is uniquely determined by $X$ up to homeomorphism since, for $G = \GL{1}(k)$, $[\Rep{\GL{1}(k)}(\Sigma)]= [\GL{1}(k)]^{2g} =(q-1)^{2g}$, where $g$ is the genus of $\Sigma$. In this way, $X$ is characterized by the genus $g$ of its normalization and the tuple of branches $(r_1, \ldots, r_l)$.

\begin{rmk}
Obviously, a simpler proof of the fact that the normalization is determined by $X$ can be obtained by applying Hurwitz formula to the branched covering $\varpi$. However, we present this proof here since it is in the line of the ``Torelli-like'' theorems, aiming to characterize the underlying manifold by means of the moduli spaces on it.
\end{rmk}

Moreover, the previous computation shows that we can focus on a very particular node-surface. Fix $g,b \geq 1$ and let $\Sigma_{g,b}$ be the node-surface whose normalization is $\Sigma_g$, the closed smooth surface of genus $g$, and with a unique conic point with $b$ branches. By (\ref{eq:rep-conic-surface}), we have that 
$
	[\Rep{G}(X)] = [\Rep{G}(\Sigma_{g,r_1 + r_2 + \ldots + r_l})]
$. 
For this reason, from now on we will focus on the node-surfaces $\Sigma_{g,b}$.

\section{Character varieties over nodefolds}\label{sec:charvar-nodefolds}

Once we have computed the virtual class of representation varieties over node-surfaces, in this section, we shall compute the virtual class of the associated character varieties. Recall that, given a complex algebraic group $G$ and a nodefold $X$, the character variety is the GIT quotient
$$
	\Char{G}(X) = \Rep{G}(X) \sslash G.
$$
Here, the action of $G$ on $\Rep{G}(X)$ is given by conjugation i.e.\ $(g \cdot \rho)(\gamma) = g \rho(\gamma)g^{-1}$ for $g \in G$, $\rho \in \Rep{G}(X)$ and $\gamma \in \pi_1(X)$.

In order to understand this quotient, we will use the theory of pseudo-quotients as developed in \cite{GP-2018b}, specially Sections 3, 4 and 5. Let us denote by $\Repred{G}(X)$ and $\Repirred{G}(X)$ the subvarieties of $\Rep{G}(X)$ of reducible and irreducible representations, respectively, and by $\Charred{G}(X) = \Repred{G}(X) \sslash G$ and $\Charirred{G}(X) = \Repirred{G}(X) \sslash G$ the corresponding character varieties. 
If $G$ is a linear algebraic group (so in particular it is affine), then, by \cite[Proposition 6.4]{GP-2018b} we have that if $G^0$ denotes the center of $G$, then $\Repirred{G}(X) \to \Repirred{G}(X) \sslash (G/G^0) = \Charirred{G}(X)$ is a free geometric quotient. Applying now \cite[Theorem 5.4]{GP-2018b} we get that $\left[\Charirred{G}(X)\right]\left[G/G^0\right] = \left[\Repirred{G}(X)\right]$.

Notice that 

Now, let us suppose that $G = \GL{n}(k)$. Consider a partition $\tau$ of $n$, that is a multiset the form $\tau = [r_1^{a_1}r_2^{a_2}\ldots r_s^{a_s}]$ with $r_i$ and $a_i$ positive integers, the $r_i$ distinct, such that $\sum_i a_ir_i = n$. We will say that a representation $\rho: \pi_1(X) \to \GL{}(k^n)$ is of type $\tau$ if the $\pi_1(X)$-module $k^n$ can be decomposed as a direct sum
$$
	k^n = \bigoplus_{i = 1}^s V_i^{a_i},
$$
where the $V_i$ are irreducible representations of dimension $\dim V_i = r_i$. An adaptation of Proposition 7.3 and Corollary 7.4 of \cite{GP-2018b} (see also \cite[Section 3]{2006.01810}) shows that any representation of $\Rep{\SL{n}(k)}(X)$ is equivalent, in the GIT quotient, to a representation of type $\tau$, and such $\tau$ is unique.

This can be used to decompose the character variety into simpler pieces. Let us denote by $\Char{}^\tau(X)$ the character variety of the representations of type $\tau$. Each representation of $\Char{}^\tau(X)$ is determined by an element of $\prod_{i} (\Charirred{\GL{r_i}(k)}(X))^{a_i}$ up to permutation of representations of the same dimension. In this way, if $S_\tau$ is the subgroup of the symmetric group $S_n$ that preserves $\tau$, we get that $(\prod_{i} (\Charirred{\GL{r_i}(k)}(X))^{a_i}, S_\tau)$ is a core (see \cite[Proposition 4.4]{GP-2018b}) for the action of $\GL{n}(k)$ on the representations of type $\tau$. Therefore, we have that
$$
	\Char{}^\tau(X) = \prod_{i=1}^s \Sym^{a_i}(\Charirred{\GL{r_i}(k)}(X)).
$$

\begin{rmk}
In particular, for $\tau = [n^1]$ we obtain the stable locus of irreducible representations, $\Char{}^\tau(X) = \Charirred{\GL{n}(k)}(X)$, and for $\tau = [1^n]$ we get the diagonal representations.
\end{rmk}
\begin{rmk}
For those familiar with Lie group theory they may see this decomposition as a result of considering the Levi subgroups of $\GL{n}(k)$ for the given partition $\tau$. \end{rmk}



Moreover, the representations of type $\tau$ form an open orbitwise-closed set of $\Char{\GL{n}(k)}(X)$ so by Theorem 4.1 of \cite{GP-2018b} we get that
$$
	\left[\Char{\GL{n}(k)}(X)\right] = \sum_\tau [\Char{}^\tau(X)] = \sum_\tau \prod_{i=1}^s \left[\Sym^{a_i}(\Charirred{\GL{r_i}(k)}(X))\right],
$$
where the sum runs over the set of all the partitions of $n$.





In the particular case of rank $n = 2$, we have that the only partitions of $2$ are $[2^1]$, corresponding to irreducible representations, and $[1^2]$, corresponding to diagonal representations. Therefore
\begin{align*}
	\left[\Char{\GL{2}(k)}(X)\right] &=  [\Char{}^{[1^2]}(X)] + [\Char{}^{[2^1]}(X)] = [\Sym^2(\Charirred{\GL{1}(k)}(X))] + [\Charirred{\GL{2}(k)}(X)] \\
	&= [(k^*)^{2(2g+b-1)}/\ZZ_2] + \frac{[\Rep{\GL{2}(k)}(X)] - [\Repred{\GL{2}(k)}(X)]}{[\PGL{2}(k)]},
\end{align*}
where we have used that $\Repirred{\GL{1}(k)}(X) = \Rep{\GL{1}(k)}(X) = (k^*)^{2g+b-1}$ if $X$ is a closed connected node-surface with normalization of genus $g$ and $b > 1$ branches.

In the case of $G = \SL{2}(k)$, the argument works verbatim but with the particularity that, now, we have to restrict to representations of unit determinant. Thus, in this case we have 
\begin{align}\label{eqn:quot}
	\left[\Char{\SL{2}(k)}(X)\right] = [(k^*)^{2g+b-1}/\ZZ_2] + \frac{[\Rep{\SL{2}(k)}(X)] - [\Repred{\SL{2}(k)}(X)]}{[\PGL{2}(k)]}.
\end{align}
At this point, the strategy to compute this expression is as follows. The virtual class of the total representation variety, $[\Rep{\SL{2}(k)}(X)]$ is the hardest part to be computed and is provided by the TQFT. On the other hand, the reducible part $[\Repred{\SL{2}(k)}(X)]$ can be computed by hand in terms of lower dimensional representations (in this case, $1$-dimensional representations which are very easy). 

The aim of the following sections is to follow this strategy. From now on, we will focus on the case $G = \SL{2}(k)$. To shorten notation, we will omit the subscript in the representation and character varieties and denote $\Rep{}(X)= \Rep{\SL{2}(k)}(X)$, $\Char{}(X)= \Char{\SL{2}(k)}(X)$ and analogous for subsequent strata and parabolic versions.


\subsection{Character varieties over orientable node-surfaces}\label{sec:charvar-nodefolds-nopar}

Let us fix $G=\SL{2}(k)$ and consider the closed connected node-surface with normalization of genus $g$ and $b > 1$ branches, $\Sigma_{g,b}$. Recall that, from equation (\ref{eq:rep-conic-surface}), we have that $[\Rep{}(\Sigma_{g,b})] = [\Rep{}(\Sigma_{g})] [\SL{2}(k)]^{b-1}$, where $\Sigma_g$ is the usual compact connected surface of genus $g$ (the normalization of $\Sigma_{g,b}$). Indeed, using the standard presentation of the fundamental group of $\Rep{}(\Sigma_{g})$ we have that
$$
	\Rep{}(\Sigma_{g,b}) = \left\{(A_1, B_1, \ldots, A_g, B_g, C_1, \ldots, C_{b-1}) \in \SL{2}(k)^{2g+b-1} \,\left|\, \prod_{i=1}^g [A_i, B_i] = I\right.\right\}.
$$
Here $[A_i, B_i] = A_iB_iA_i^{-1}B_i^{-1}$ denotes the group commutator. Let $\mathcal{A} \in \SL{2}(k)^{2g+b-1}$ be a tuple of upper-triangular matrices, say
$$
	\mathcal{A}= \left(
	\begin{pmatrix}
		\lambda_1 & \alpha_1 \\
		0 & \lambda_1^{-1}
	\end{pmatrix},
	\begin{pmatrix}
		\mu_1 & \beta_1 \\
		0 & \mu_1^{-1}
	\end{pmatrix}, \ldots,
	\begin{pmatrix}
		\lambda_{g} & \alpha_{g} \\
		0 & \lambda_{g}^{-1}
	\end{pmatrix},
	\begin{pmatrix}
		\mu_{g} & \beta_{g} \\
		0 & \mu_{g}^{-1}
	\end{pmatrix}, 
	\begin{pmatrix}
		\eta_{1} & \gamma_{1} \\
		0 & \eta_{1}^{1}
	\end{pmatrix}, \ldots,
	\begin{pmatrix}
		\eta_{b-1} & \gamma_{b-1} \\
		0 & \eta_{b-1}^{-1}
	\end{pmatrix}
	\right)
$$
with $\lambda_i, \mu_i, \eta_i \in k^* = k - \left\{0\right\}$ and $\alpha_i, \beta_i, \gamma_i \in k$. A straightforward computation shows that $\mathcal{A} \in \Rep{}(\Sigma_{g,b})$ if and only if
\begin{equation*}\label{eq:cond-upper}
	\sum_{i=1}^g \lambda_i\mu_i \left[\left(\mu_i-\mu_i^{-1}\right)\beta_i - \left(\lambda_i-\lambda_i^{-1}\right)\alpha_i\right]= 0.
\end{equation*}
We will denote by $\pi \subseteq k^{2g+b-1}$ the $(\alpha_i, \beta_i)$-plane defined by the previous equation for fixed $(\lambda_i, \mu_i)$.

Now, let us stratify $\Repred{}(\Sigma_{g,b})$ as follows.
\begin{itemize}
	\item $\XI{\Rep{}(\Sigma_{g,b})}$ is the set of tuples with all the matrices equal to $\pm \Id$ or, equivalently, the set of completely reducible representations with with the images of the generators of trace $\pm 2$. Hence, $\XI{\Rep{}(\Sigma_{g,b})}$ is a set of $2^{2g+b-1}$ matrices.
	\item $\XPh{\Rep{}(\Sigma_{g,b})}$ is the set of tuples with matrices with trace $\pm 2$ that are not completely reducible. Given $A \in \XPh{\Rep{}(\Sigma_{g,b})}$, let
$$
    \left(\begin{pmatrix} \epsilon_1 & a_1 \\ 0 & \epsilon_1\end{pmatrix}, \begin{pmatrix} \epsilon_2 & a_2 \\ 0 & \epsilon_2\end{pmatrix}, \ldots, \begin{pmatrix}\epsilon_{2g+b-1} & a_{2g+b-1} \\ 0 & \epsilon_{2g+b-1}\end{pmatrix}\right)
$$ be the element of conjugate to $A$ with $\epsilon_i = \pm 1$ and $a_i \in k$ not all zero. Observe that such an element is unique up to simultaneous rescalling of the off-diagonal entries $a_i$. Thus, the $\SL{2}(k)$-orbit of $A$, $[A]$, is the set of reducible representations $(B_1, \ldots, B_{2g+b-1}) \in \Xf{n}$ with a double eigenvalue such that, in their upper triangular form,
$$
    \left(\begin{pmatrix} \epsilon_1 & b_1 \\ 0 & \epsilon_1\end{pmatrix},\begin{pmatrix} \epsilon_2 & b_2 \\ 0 & \epsilon_2\end{pmatrix}, \ldots, \begin{pmatrix} \epsilon_{2g+b-1} & b_{2g+b-1} \\ 0 & \epsilon_{2g+b-1}\end{pmatrix}\right)
$$
there exists $\lambda \neq 0$ such that $(a_1, \ldots, a_{2g+b-1}) = \lambda(b_1, \ldots, b_{2g+b-1})$. Then, taking $\lambda \to 0$, we find that the closure of the orbit, $\overline{[A]}$, is precisely the set of reducible representations with double eigenvalue such that their off-diagonal entries satisfy $(a_1, \ldots, a_{2g+b-1}) = \lambda(b_1, \ldots, b_{2g+b-1})$ for some $\lambda \in k$. In particular, for $\lambda = 0$ we get that the special element $(\epsilon_1\Id, \ldots, \epsilon_{2g+b-1}\Id) \in \overline{[A]}$.

	Therefore, the tuple $(a_1, \ldots, a_{2g+b-1}) \in k^{2g+b-1}-\left\{0\right\}$ determines the diagonal form up to projectivization. The stabilizer of a Jordan type matrix under the action of $\SL{2}(k)$ by conjugation is $k$, so the orbit of an element is $\SL{2}(k)/k$. Hence, we obtain a regular fibration
	$$
		k^* \longrightarrow \SL{2}(k)/k \times \left\{\pm 1\right\}^{2g+b-1} \times \left(k^{2g+b-1}-\left\{0\right\}\right) \longrightarrow \XPh{\Rep{}(\Sigma_{g,b})}.
	$$
Observe that this fibration is locally trivial in the Zariski topology so this fibration has trivial monodromy and, thus
\begin{align*}
    \left[\XPh{\Rep{}(\Sigma_{g,b})}\right] &= \left[\left\{\pm 1\right\}^{2g+b-1}\right] \left[\PP^{2g+b-2}\right] \left[\SL{2}(k)/\Stab\,J_+\right] \\&= 2^{{2g+b-1}}(q^2-1)\frac{q^{2g+b-1} -1}{q-1}.
\end{align*}
Recall that, as introduced in Section \ref{sec:grothendieck-ring}, we denote the Lefschetz motive as $q = [k]$.
	\item $\XDh{\Rep{}(\Sigma_{g,b})}$ is the set of completely reducible representations with images not all equal to $\pm \Id$. Given $A$ in this stratum, let
$$
    \left(\begin{pmatrix} \lambda_1 & 0 \\ 0 & \lambda_1^{-1}\end{pmatrix}, \ldots, \begin{pmatrix}\mu_{g} & 0 \\ 0 & \mu_{g}^{-1}\end{pmatrix}, \begin{pmatrix}\eta_{1} & 0 \\ 0 & \eta_{1}^{-1}\end{pmatrix}, \ldots, \begin{pmatrix}\eta_{b-1} & 0 \\ 0 & \eta_{b-1}^{-1}\end{pmatrix}\right)
$$ be the element conjugate to $A$ with $(\lambda_i, \mu_i, \eta_i) \in (k^*)^{2g+b-1}$ and not all equal to $\pm 1$. The stabilizer of a diagonal matrix is $k^*$ so the orbit of this representation is $\SL{2}(k)/k^*$. This canonical diagonal form of an element of $\XDh{\Rep{}(\Sigma_{g,b})}$ is unique up to simultaneous permutation of the eigenvalues, so we have a double covering
	$$
		\SL{2}(k)/k^* \times \left((k^*)^{2g+b-1}-\left\{(\pm 1, \ldots, \pm 1)\right\}\right) \longrightarrow \XDh{\Rep{}(\Sigma_{g,b})}.
	$$
Therefore, we obtain that
$$
	\XDh{\Rep{}(\Sigma_{g,b})} = \frac{\SL{2}(k)/k^* \times \left[(k^*)^{2g+b-1}-\left\{(\pm 1, \ldots, \pm 1)\right\}\right]}{\ZZ_2}.
$$
Using \cite[Remark 5.3]{GP-2018b}, it virtual class is
$$
	\left[\XDh{\Rep{}(\Sigma_{g,b})}\right] = \frac{q^3-q}{2} \left((q-1)^{{2g+b-2}} + (q+1)^{2g+b-2}\right) - 2^{2g+b-1}q^2.
$$
	\item $\XTilde{\Rep{}(\Sigma_{g,b})}$ is the set of reducible representations, not completely reducible, so that not all its matrices have double eigenvalue. In this case, any element is conjugated to one of the form
$$
	\left(
	\begin{pmatrix}
		\lambda_1 & \alpha_1 \\
		0 & \lambda_1^{-1}
	\end{pmatrix},
	\begin{pmatrix}
		\mu_1^{-1} & \beta_1 \\
		0 & \mu_1
	\end{pmatrix}, \ldots,
	\begin{pmatrix}
		\lambda_{g} & \alpha_{g} \\
		0 & \lambda_{g}
	\end{pmatrix},
	\begin{pmatrix}
		\mu_{g} & \beta_{g} \\
		0 & \mu_{g}
	\end{pmatrix}, 	\begin{pmatrix}
		\eta_{1} & \gamma_{1} \\
		0 & \eta_{1}^{-1}
	\end{pmatrix}, \ldots, 	\begin{pmatrix}
		\eta_{g} & \gamma_{b-1} \\
		0 & \eta_{b-1}^{-1}
	\end{pmatrix}
	\right)
$$
with $(\lambda_i, \mu_i,\eta_{i}) \in (k^*)^{2g + b-1} - \left\{(\pm 1, \ldots, \pm 1)\right\}$ and $(\alpha_i, \beta_i, \gamma_i) \in \pi \times k^{b-1} - l$ where $l$ is the line spanned by $(\lambda_1 - \lambda_1^{-1}, \mu_1 - \mu_1^{-1}, \ldots, \lambda_g - \lambda_g^{-1}, \mu_g - \mu_g^{-1}, \ldots, \eta_{b-1}-\eta_{b-1}^{-1})$. The action of $\PGL{2} = \SL{2}/\{\pm \Id\}$ by conjugation on these elements is free, and the canonical form is determined up to action of an upper-triangular matrix of $k \times k^*$. Thus, we have a fibration
$$
	 k \times k^* \longrightarrow \PGL{2} \times \Omega \longrightarrow \XTilde{\Rep{}(\Sigma_{g,b})},
$$
where $\Omega$ is a locally trivial fibration $\left(\pi \times k^{b-1} - l\right) \to \Omega \to (k^*)^{2g+b-1} - \left\{(\pm 1, \ldots, \pm 1)\right\}$. Using that $\left[\pi\right] = q^{2g-1}$ and $[l] = q$, we get that the virtual class of $\Omega$ is $[\Omega] = \left((q-1)^{2g+b-1}-2^{2g+b-1}\right)\left(q^{2g+b-2}-q\right)$ and the virtual class of this stratum is
$$
	\left[\XTilde{\Rep{}(\Sigma_{g,b})}\right] = \frac{q^3-q}{(q-1)q} \left((q-1)^{2g + b-1} - 2^{2g + b-1}\right)\left(q^{2g+b-2}-q\right).
$$
\end{itemize}

Therefore, putting all together we have that
\begin{align*}
	\left[\Repred{}(\Sigma_{g,b})\right] &= (q+1) (q-1)^{2g + b-1}\left(q^{2g+b-2}-q\right) \\
	&+ \frac{q^3-q}{2} \left((q-1)^{2g+b-2} + (q+1)^{2g+b-2}\right) - 2^{2g+b-1}(q^2-1).
\end{align*}

Using the results of Section 5.4 of \cite{GP-2018} (see also \cite[Proposition 11]{MM} for the weaker case of $E$-polynomials) and equation (\ref{eq:rep-conic-surface}), we know that virtual class of the total representation variety is
\begin{align*}
	\left[\Rep{}(\Sigma_{g,b})\right] =&\, \left[\Rep{}(\Sigma_{g})\right][\SL{2}(k)]^{b-1} = \left(2^{2g - 1} {\left(q - 1\right)}^{2g - 1} {\left(q + 1\right)}
q^{2g - 1}\right. \\
&+ 2^{2g - 1} {\left(q + 1\right)}^{2g - 1}
{\left(q - 1\right)} q^{2g - 1} + {{\left(q + q^{2 \,
g-1}\right)} {\left(q^2 - 1\right)}^{2 \,
g - 1}}\\
&\left.+ \frac{1}{2} \, {\left(q +
1\right)}^{2g - 1} {\left(q - 1\right)}^{2} q^{2g - 1} +
\frac{1}{2} \, {\left(q - 1\right)}^{2g - 1} {\left(q + 1\right)}
{\left(q - 3\right)} q^{2g - 1}\right)(q^3-q)^{b-1}.
\end{align*}
Moreover, for the quotient of the reducible locus \cite[Corollary 7.5]{GP-2018b} shows that
$$
[(k^*)^{2g+b-1} \sslash \ZZ_2] = \frac{1}{2} \left((q-1)^{2g+b-1} + (q+1)^{2g+b-1}\right).
$$
So, using formula (\ref{eqn:quot}), we finally find that


\begin{align}\label{eq:nopar}
	\left[\Char{}(\Sigma_{g,b})\right] &= \frac{1}{2 q^3 \left(q^2-1\right)^3}\left[q^8 \left((q-1)^b+1\right) (q-1)^{2 g}+q^7 \left((q-1)^b+1\right) (q-1)^{2 g}\right.\notag\\
	&-2 q^6 \left((q-1)^b+1\right) (q-1)^{2 g}-2 q^5
   \left((q-1)^b+1\right) (q-1)^{2 g}\notag\\
   &+q^4 \left((q-1)^b+1\right) (q-1)^{2 g}+q^3 \left((q-1)^b+1\right) (q-1)^{2 g}\notag\\
   &+(q-1)^3 q^3
   \left((q+1)^b-1\right) (q+1)^{2 g+2}+\left(q \left(q^2-1\right)\right)^b \left((q+1)^2 (q-1)^{2 g} \left(4^g+q-3\right) q^{2g} \right.\\
   &\left.+(q-1)^2 (q+1)^{2 g} \left(4^g+q-1\right) q^{2 g}+2 \left(q^2-1\right)^{2 g} \left(q^{2 g}+q^2\right)\right)\notag\\
   &\left.-\left(q^2-1\right)q^{2 g+1} \left(2 \left(4^g \left(q^4+q^2+1\right)+(q+1)^2 (q-1)^{2 g+1}\right)-q^2 3\cdot 2^{2 g+1}\right)\right]\notag
		\end{align}

\begin{rmk}
Observe that the calculation above is only valid for $g \geq 1$. This is due to the fact that, for the stratum $\XTilde{\Rep{}(\Sigma_{g,b})}$, the set of feasible eigenvalues and antidiagonal elements are required to lie in a hyperplane. However, this equation vanishes for $g=0$ but, in that case, we only need to consider the quotient of a free group, as studied in Section 7.1 of \cite{GP-2018b} or \cite{florentino2019polynomials}.
\end{rmk}

\subsection{The parabolic case with punctures of Jordan type}\label{sec:parabolic-jordan}
Now, let us address the computation of the virtual class of the character variety in the parabolic case. In this section, we will focus on the case in which the parabolic structure $Q$ only contains punctures of Jordan type
$$
	J_+ = \begin{pmatrix}
	1 & 1\\
	0 & 1\\
\end{pmatrix},
$$
that is $Q = \left\{(p_1, [J_+]), \ldots, (p_s, [J_+])\right\}$ with $p_1, \ldots, p_s \in \Sigma_{g,b}$ different points. Observe that the stabilizer of $J_+$ by conjugation is $\Stab J_+ = k$. In particular, $[\SL{2}(k)/\Stab J_+] = q^2-1$.

As before, consider a general representation of the form
$$
	A = \left(
	\begin{pmatrix}
		\lambda_1 & \alpha_1 \\
		0 & \lambda_1^{-1}
	\end{pmatrix},
	\begin{pmatrix}
		\mu_1 & \beta_1 \\
		0 & \mu_1^{-1}
	\end{pmatrix}, \ldots,
\begin{pmatrix}
		\eta_{1} & \gamma_{1} \\
		0 & \eta_{1}^{-1}
	\end{pmatrix}, \ldots, 	\begin{pmatrix}
		\eta_{b-1} & \gamma_{b-1} \\
		0 & \eta_{b-1}^{-1}
	\end{pmatrix},
	\begin{pmatrix}
		1 & c_1 \\
		0 & 1
	\end{pmatrix}, \ldots,
	\begin{pmatrix}
		1 & c_s \\
		0 & 1
	\end{pmatrix}
	\right),
$$
with $\lambda_i, \mu_i, \eta_i \in k^*$, $\alpha_i, \beta_i, \gamma_i \in k$ and $c_i \in k^*$. Then, $A \in \Rep{}(\Sigma_{g,b}, Q)$ if and only if
\begin{equation}
	\sum_{i=1}^g \lambda_i\mu_i \left[\left(\mu_i-\mu_i^{-1}\right)\beta_i - \left(\lambda_i-\lambda_i^{-1}\right)\alpha_i\right] + \sum_{i=1}^s c_i= 0.
\end{equation}
For $(\lambda_i, \mu_i)$ fixed, let us denote by $\pi_s \subseteq k^{2g+s}$ the $(\alpha_i, \beta_i, c_i)$-plane defined by the previous equation. In order to compute the virtual class of $\pi_s$ observe that, solving for $c_s$, we observe that $\pi_s = k^{2g} \times (k^*)^{s-1} - \pi_{s-1}$. Using as base case that $\pi_1$ is $k^{2g}$ minus a hyperplane, we have
\begin{align*}
    [\pi_s] &= q^{2g}(q-1)^{s-1} - [\pi_{s-1}] \\
    &= q^{2g}\sum_{k=1}^{s} (-1)^{k+1}(q-1)^{s-k} + (-1)^s q^{2g-1} = q^{2g-1}(q-1)^s.
\end{align*}

Now, we analyze each stratum of possible reducible representations as above. Using the previous notations for the strata we have the following.

\begin{itemize}
	\item $\XI{\Rep{}(\Sigma_{g,b}, Q)} = \emptyset$ since the punctures cannot be $\pm \Id$.
	\item $\XPh{\Rep{}(\Sigma_{g,b}, Q)}$. As in the previous case, an element $A \in \XPh{\Rep{}(\Sigma_{g,b}), Q}$ is conjugated to one of the form
$$
 \left(
	\begin{pmatrix}
		\epsilon_1 & \alpha_1 \\
		0 & \epsilon^{-1}
	\end{pmatrix},
	\begin{pmatrix}
		\epsilon_2 & \beta_1 \\
		0 & \epsilon_2^{-1}
	\end{pmatrix}, \ldots,
\begin{pmatrix}
		\epsilon_{2g+1} & \gamma_{1} \\
		0 & \epsilon_{2g+1}^{-1}
	\end{pmatrix}, \ldots,
	\begin{pmatrix}
		1 & c_1 \\
		0 & 1
	\end{pmatrix}, \ldots,
	\begin{pmatrix}
		1 & c_s \\
		0 & 1
	\end{pmatrix}
	\right),
$$
where $\epsilon_i = \pm 1$ and $c_i \neq 0$. Moreover, this element has to lie in a subset of the plane $\pi_s$, which in this case amounts to the hyperplane
$$
    \tilde{\pi}_s = \left\{(c_1, \ldots, c_s) \in (k^{*})^s\;\;\left|\,\sum_{i = 1}^s c_i = 0\right.\right\}.
$$
To compute the virtual class of this space, observe that $\tilde{\pi}_s = (k^*)^{s-1} - \tilde{\pi}_{s-1}$. Therefore, using the base case $\tilde{\pi}_1 = \emptyset$, we have
\begin{align*}
    [\tilde{\pi}_s] &= (q-1)^{s-1} - [\pi_{s-1}] = \sum_{k=1}^{s-1} (-1)^{k+1}(q-1)^{s-k} = (-1)^s \left(\frac{(1-q)^s - 1}{q} + 1\right).
\end{align*}

As in the non-parabolic case, such element is unique up to re-scalling of the off-diagonal entries. Therefore, we obtain a regular fibration trivial in the Zariski topology
	$$
		k^* \longrightarrow \SL{2}(k)/\Stab\,J_+ \times \left\{\pm 1\right\}^{2g+b-1} \times \left(k^{2g+b-1}\times \tilde{\pi}_s\right) \longrightarrow \XPh{\Rep{}(\Sigma_{g,b},Q)}.
	$$
Hence, taking virtual classes we obtain
\begin{align*}
    \left[\XPh{\Rep{}(\Sigma_{g,b}, Q)}\right] &= \left[\left\{\pm 1\right\}^{2g+b-1}\right] \frac{\left[k^{2g+b-1}\right] \left[\tilde{\pi}_s\right]}{[k] - 1} \left[\SL{2}(k)/\Stab\,J_+\right] \\&= 2^{{2g+b-1}}(q^2-1)\frac{q^{2g+b-1}}{q-1}\left((-1)^s \left(\frac{(1-q)^s - 1}{q} + 1\right)\right).
\end{align*}	
	\item $\XDh{\Rep{}(\Sigma_{g,b}, Q)} = \emptyset$ since the holonomies of the punctures are not diagonalizable.
	\item $\XTilde{\Rep{}(\Sigma_{g,b}, Q)}$. In this case, any element is conjugated to one of the form
$$
	\left(
	\begin{pmatrix}
		\lambda_1 & \alpha_1 \\
		0 & \lambda_1^{-1}
	\end{pmatrix}, \ldots,
	\begin{pmatrix}
		\mu_g & \beta_n \\
		0 & \mu_g^{-1}
	\end{pmatrix},
	\begin{pmatrix}
		\eta_{1} & \gamma_{1} \\
		0 & \eta_{1}^{-1}
	\end{pmatrix},
	\ldots,
	\begin{pmatrix}
		\eta_{b-1} & \gamma_{b-1} \\
		0 & \eta_{b-1}^{-1}
	\end{pmatrix}
	\begin{pmatrix}
		1 & c_{1} \\
		0 & 1
	\end{pmatrix}, \ldots,
	\begin{pmatrix}
		1 & c_{s} \\
		0 & 1
	\end{pmatrix}
	\right)
$$
with $(\lambda_1, \ldots, \mu_g, \eta_1, \ldots, \eta_{g-1}) \in (k^*)^{2g+b-1} - \left\{(\pm 1, \ldots, \pm 1)\right\}$, $\gamma_i \in k$ and the vector of remaining off-diagonal entries $(\alpha_i, \beta_i, c_i) \in  \pi_s$.

Notice that, now, we no longer have a condition of simultaneously non-vanishing off-diagonal entries, since the punctures are non-diagonalizable. Therefore, we have a fibration
$$
	k \times k^* \longrightarrow \PGL{2} \times \Omega \longrightarrow \XTilde{\Rep{}(\Sigma_{g,b}, Q)},
$$
where $\pi_s \to \Omega \to \left((k^*)^{2g+b-1} - \left\{(\pm 1, \ldots, \pm 1)\right\}\right) \times k^{b-1}$ is a locally trivial fibration. Thus,
$$
	\left[\XTilde{\Rep{}(\Sigma_{g,b}, Q)}\right] = \frac{q^3-q}{(q-1)q} \left((q-1)^{2g+b-1} - 2^{2g+b-1}\right)q^{b-1}\left(q^{2g-1}(q-1)^s\right).
$$
\end{itemize}
Therefore, putting all together, we obtain that
\begin{align*}
	\left[\Repred{}(\Sigma_{g,b}, Q)\right]&= \frac{1}{2q-2}\left(q^{b+2 g-2} \left(\left(q^2-1\right) 2^{b+2 g} \left((-1)^s \left((1-q)^s+q-1\right)-(q-1)^s\right)\right.\right.\\
	&+\left. \left.2 (q+1) (q-1)^{b+2 g+s}\right)\right)
\end{align*}

In \cite[Theorem 5.10]{GP-2018b}, it is proven that the virtual class of the total representation variety of the normalization is
\begin{align*}
\left[\Rep{}(\Sigma_{g}, Q)\right] =& \,{\left(q^2 - 1\right)}^{2g + s - 1} q^{2g - 1} +
\frac{1}{2} \, {\left(q -
1\right)}^{2g + s - 1}q^{2g -
1}(q+1){\left({2^{2g} + q - 3}\right)} 
\\ &+ \frac{\left(-1\right)^{s}}{2} \,
{\left(q + 1\right)}^{2g + s - 1} q^{2g - 1} (q-1){\left({2^{2g} +q -1}\right)}.
\end{align*}
With this result, we can obtain the irreducible character variety as
$$
	\left[\Charirred{}(\Sigma_{g,b}, Q)\right] = \frac{\left[\Rep{}(\Sigma_{g}, Q)\right](q^3-q)^{b-1} - \left[\Repred{}(\Sigma_{g,b}, Q)\right]}{q^3-q}
$$

Up to this point, the calculation in this parabolic case is analogous to the non-parabolic setting. Nevertheless, for the reducible locus the situation turns completely different from the previous one. Recall that $\XI{\Rep{}(\Sigma_{g,b}, Q)} = \emptyset$ so there are not completely reducible representations (i.e.\ of type $\tau = [1^2]$). Precisely for this reason, the action of $\SL{2}(k)$ on $\Repred{}(\Sigma_{g,b}, Q)$ is closed (see Section 8.1 of \cite{GP-2018b}). However, this action is not globally free so we have to distinguish between the two strata of $\Repred{}(\Sigma_{g,b}, Q)$:
	\begin{itemize}
		\item $\XTilde{\Rep{}(\Sigma_{g,b}, Q)}$. Here, the action of $\PGL{2}(k)$ is free so, by \cite[Corollary 5.5]{GP-2018b}, we have
		\begin{align*}
			\left[\XTilde{\Rep{}(\Sigma_{g,b}, Q)} \sslash \SL{2}(k)\right] &= \frac{\left[\XTilde{\Rep{}(\Sigma_{g,b}, Q)}\right]}{q^3-q} \\
			&=  \left((q-1)^{2g+b-1} - 2^{2g+b-1}\right)q^{2g+b-3}(q-1)^{s-1}.
		\end{align*}
		\item $\XPh{\Rep{}(\Sigma_{g,b}, Q)}$. Here, the action of $\PGL{2}(k)$ is not free, but it has stabilizer isomorphic to $\Stab\,J_+ \cong k$. Hence, the GIT quotient
		$
			\XPh{\Rep{}(\Sigma_{g,b}, Q)} \to \XPh{\Rep{}(\Sigma_{g,b}, Q)} \sslash \SL{2}(k)
		$
is a locally trivial fibration with fiber $\SL{2}/\Stab\,J_+$ and trivial monodromy. Thus, we have that
		$$
			\left[\XPh{\Rep{}(\Sigma_{g,b}, Q)} \sslash \SL{2}(k)\right] = \frac{\left[\XPh{\Rep{}(\Sigma_{g,b}, Q)}\right]}{q^2-1} = (-1)^s2^{{2g+b-1}}\frac{q^{2g+b-1}}{q-1}\left(\frac{(1-q)^s - 1}{q} + 1\right).
		$$
	\end{itemize}

Summarizing, the analysis above shows that

\begin{align}\label{eq:par_jordan}
\left[\Char{}(\Sigma_{g,b}, Q)\right] &=\frac{q^{2 g-3}}{2 \left(q^2-1\right)^3}\left[\left(4^g-3\right) \left(q \left(q^2-1\right)\right)^b (q-1)^{2 g+s}\right.\notag\\
&\left.+\left(q \left(q^2-1\right)\right)^b \left(q \left(4^g(q+2)+q^2-q-5\right) (q-1)^{2 g+s}+2 \left(q^2-1\right)^{2 g+s}\right.\right.\notag\\
&\left.\left.+q (-1)^s \left(4^g (q-2)+q^2-3 q+3\right) (q+1)^{2g+s}+\left(4^g-1\right) (-1)^s (q+1)^{2 g+s}\right)\right.\\
&\left.+\left(q^2-1\right) (-1)^s
   2^{b+2 g} q^b \left(q^5-q^4+q^3
  +\left(q^2-1\right)^2 (1-q)^s+2 q^2+q-1\right)\right.\notag\\
  &\left. +\left(q^2-1\right) (-1)^{s+1} q^{b+3} 3\cdot 2^{b+2 g}\right]\notag
\end{align}

\subsection{The parabolic case with diagonal punctures}\label{sec:parabolic-diagonal}
In this section, we shall study the case of puctures of semi-simple type. That is, punctures whose holonomy is conjugated to
$$
D_\xi = \begin{pmatrix}
	\xi & 0\\
	0 & \xi^{-1}\\
\end{pmatrix},
$$
for $\xi \in k^* - \left\{\pm 1\right\}$. Observe that $D_\xi$ and $D_{\xi^{-1}}$ are conjugated so the conjugacy class of $D_\xi$, $[D_\lambda]$, is determined by the trace $\tr(D_\xi) = \xi + \xi^{-1}$. 
Particularly, in this section we will focus on a parabolic structure of the form $Q = \left\{(p_1, [D_{\xi_1}]), \ldots, (p_s, [D_{\xi_s}])\right\}$, for fixed $\xi_1, \ldots, \xi_s \in k^* - \left\{\pm 1\right\}$. It will be useful to consider the set
$$
	\Lambda_\pm = \left\{(\epsilon_1, \ldots, \epsilon_s) \in \set{1, -1}^s \,\left|\,\xi_1^{\epsilon_1}\cdots \xi_s^{\epsilon_s} = \pm1\right. \right\},
$$
and the integers $\alpha_\pm = \frac{1}{2}\left|\Lambda_\pm\right|$.

In order to understand the action of $\SL{2}(k)$ on the representation variety $\Rep{}(\Sigma_g, Q)$, consider a tuple of matrices $A = (A_1, B_1, \ldots, A_g, B_g, C_1, \ldots, C_{b-1}, P_1, \ldots, P_s)$ of the form
\begin{equation}\label{eq:form-diagonal-punct}
	\left(
	\begin{pmatrix}
		\lambda_1 & \alpha_1 \\
		0 & \lambda_1^{-1}
	\end{pmatrix},
	\begin{pmatrix}
		\mu_1 & \beta_1 \\
		0 & \mu_1^{-1}
	\end{pmatrix}, \ldots,
\begin{pmatrix}
		\eta_{1} & \gamma_{1} \\
		0 & \eta_{1}^{-1}
	\end{pmatrix}, \ldots, 
	\begin{pmatrix}
		\vartheta_1 & c_1 \\
		0 & \vartheta_1^{-1}
	\end{pmatrix}, \ldots,
	\begin{pmatrix}
		\vartheta_s & c_s \\
		0 & \vartheta_s^{-1}
	\end{pmatrix}
	\right),
\end{equation}
with $\lambda_i, \mu_i, \nu_i \in k^*$ and $\alpha_i, \beta_i, \gamma_i, c_i \in k$. Then, we have that
$$
	\prod_{i=1}^g[A_i, B_i] \prod_{k=1}^s P_k =
	\begin{pmatrix}
	{\displaystyle\vartheta_1\cdots \vartheta_s} & {\displaystyle \sum_{i=1}^g \lambda_i\mu_i \left[\left(\mu_i-\mu_i^{-1}\right)\beta_i - \left(\lambda_i-\lambda_i^{-1}\right)\alpha_i\right] + \sum_{k=1}^s \left(\prod_{j \neq k}\vartheta_j\right) c_k}
	\\ 0 & {\displaystyle\vartheta_1^{-1}\cdots \vartheta_s^{-1}}
	\end{pmatrix}
$$
Therefore, $A \in \Rep{}(\Sigma_{g,b}, Q)$ if and only if the following system of equations holds
\begin{equation}\label{eq:cond-upper:parabolic-diagonal}
	\left\{\begin{matrix}
	{\displaystyle \vartheta_i + \vartheta_i^{-1} = \xi_i + \xi_i^{-1}}, \quad
	{\displaystyle \vartheta_1\cdots \vartheta_s = 1}, \\
	{\displaystyle \sum_{i=1}^g \lambda_i\mu_i \left[\left(\mu_i-\mu_i^{-1}\right)\beta_i - \left(\lambda_i-\lambda_i^{-1}\right)\alpha_i\right] + \sum_{k=1}^s \vartheta_k^{-1} c_k= 0}.
	\end{matrix}\right.
\end{equation}
In particular, the first line imposes that $(\vartheta_1, \ldots, \vartheta_s) = (\xi_1^{\epsilon_1}, \ldots, \xi_s^{\epsilon_s})$ for some $(\epsilon_1, \ldots, \epsilon_s) \in \Lambda_+$. Let us denote by $\pi$ the $(\alpha_i, \beta_i, c_i)$-hyperplane of $k^{2g+s}$ given by the second line of equation (\ref{eq:cond-upper:parabolic-diagonal}). Observe that we can solve for $c_s$ so $[\pi] = q^{2g+s-1}$.


For the quotient $\Repred{}(\Sigma_{g,b}, Q) \sslash \SL{2}(k)$, as in the non-parabolic case, the diagonal matrices with the action of $\ZZ_2$ by interchanging the eigenvalues form a core for the action. In this case, these diagonal matrices are $(k^*)^{2g+b-1} \times \Lambda_+$ and $\ZZ_2$ acts on $\Lambda$ by $(\epsilon_1,\ldots, \epsilon_s) \mapsto (-\epsilon_1,\ldots, -\epsilon_s)$. Hence, we obtain
$$
	\left[\Repred{}(\Sigma_{g,b}, Q) \sslash \SL{2}(k)\right] = \left[(k^*)^{2g+b-1} \times \Lambda_{+} / \ZZ_2\right] = \alpha_+(q-1)^{2g+b-1}.
$$

The calculation of the virtual class $\left[\Repred{}(\Sigma_{g,b}, Q)\right]$ can be done by stratifying it as in the non-parabolic case, but taking into account equations (\ref{eq:cond-upper:parabolic-diagonal}).

\begin{itemize}
	\item $\XI{\Rep{}(\Sigma_{g,b}, Q)} = \emptyset$ since the punctures cannot be $\pm \Id$.
	\item $\XPh{\Rep{}(\Sigma_{g,b}, Q)} = \emptyset$ since the puctures do not have eigenvalues $\pm 1$.
	\item For $\XDh{\Rep{}(\Sigma_{g,b}, Q)}$ we have that every element is conjugate to one of the form
$$
	\left(
	\begin{pmatrix}
		\lambda_1 & 0 \\
		0 & \lambda_1^{-1}
	\end{pmatrix}, \ldots,
	\begin{pmatrix}
		\mu_g & 0 \\
		0 & \mu_g^{-1}
	\end{pmatrix},
	\begin{pmatrix}
		\eta_{1} & 0 \\
		0 & \eta_{1}^{-1}
	\end{pmatrix},\ldots,
	\begin{pmatrix}
		\eta_{b-1} & 0 \\
		0 & \eta_{b-1}^{-1}
	\end{pmatrix},
	\begin{pmatrix}
		\xi_1^{\epsilon_1} & 0 \\
		0 & \xi_1^{-\epsilon_1}
	\end{pmatrix}, \ldots,
	\begin{pmatrix}
		\xi_s^{\epsilon_s} & 0 \\
		0 & \xi_s^{-\epsilon_s}
	\end{pmatrix}
	\right),
$$
with $(\lambda_1, \ldots, \mu_g, \eta_1, \ldots, \eta_{b-1}) \in (k^*)^{2g+b-1}$ and $(\epsilon_1, \ldots, \epsilon_s) \in \Lambda_+$. This representation is unique up to permutation of the eigenvalues, so we have a double covering
$$
	\left(\SL{2}(k)/{k^*}\right) \times (k^*)^{2g+b-1} \times \Lambda \longrightarrow \XDh{\Rep{}(\Sigma_{g,b}, Q)}.
$$
Therefore, as shown in \cite[Remark 5.3]{GP-2018b}, we obtain
$$
	\left[\XDh{\Rep{}(\Sigma_{g,b}, Q)}\right] = 2\alpha_+(q^2+q)(q-1)^{2g+b-1}.
$$
	\item For $\XTilde{\Rep{}(\Sigma_{g,b}, Q)} \subseteq \Repred{}(\Sigma_g, Q)$, we have that any element is conjugated to one of the form (\ref{eq:form-diagonal-punct}). According to equations (\ref{eq:cond-upper:parabolic-diagonal}), this implies that $(\lambda_i,\mu_i, \eta_i) \in (k^*)^{2g+b-1}$, $(\vartheta_1, \ldots, \vartheta_s) = (\xi_1^{\epsilon_1}, \ldots, \xi_s^{\epsilon_s})$ for some $(\epsilon_1, \ldots, \epsilon_s) \in \Lambda_+$, and the off-diagonal entries $(\alpha_i,\beta_{i}, c_i) \in \pi$.
	
In order to compute its virtual class, we have a fibration
$$
	k^* \times k \longrightarrow \PGL{2}(k) \times \Omega \longrightarrow \XTilde{\Rep{}(\Sigma_{g,b}, Q)}.
$$
Here, $\Omega = \left((k^*)^{2g+b-1} \times \Lambda_+\right) \times \left(\pi \times k^{b-1} -  \ell\right)$ with $\ell$ the line spanned by $(\lambda_1- \lambda_1^{-1}, \ldots, \mu_g - \mu_g^{-1}, \eta_1 - \eta_1^{-1}, \ldots, \eta_s - \eta_s^{-1}, \lambda_1 - \lambda_1^{-1}, \ldots, \lambda_s - \lambda_s^{-1})$. Therefore, the virtual class of $\Omega$ is $[\Omega] = \alpha_+(q-1)^{2g+b-1}(q^{2g+s-1}q^{b-1}-q)$ and we have
$$
	\left[\XTilde{\Rep{}(\Sigma_{g,b}, Q)}\right] = 4\alpha_+\frac{q^3-q}{(q-1)q} (q-1)^{2g+b-1}\left(q^{2g+b+s-2}-q\right).
$$
\end{itemize}
Hence, putting all the computations together we get
$$
\left[\Repred{}(\Sigma_{g,b}, Q)\right]  =\frac{1}{q^2}\left(2 (q+1) \alpha _{+} (q-1)^{b+2 g-1} \left(q^3-2 q^{b+2 g+s}\right)\right)
$$

In \cite[Theorem 5.6]{gonzalez2020virtual} it is shown that the total representation variety of the normalization is
\begin{align*}
	\left[\Rep{}(\Sigma_g, Q)\right] =\, & q^{2g+s-1}(q - 1)^{2g-1}(q + 1)(2^{2g+ s - 1} - 2^s + (q + 1)^{2g+s-2} \\
	&+ q^{2-2g-s}(q + 1)^{2g+s-2}) + {\cI}_0(\xi_1, \ldots, \xi_s),
\end{align*}
where the interaction term is given by
\begin{align*}
{\cI}_0(\xi_1, \ldots, \xi_s) =\, & q^{s-1}(q - 1)^{2g-1}(q+1)(\alpha_+ + \alpha_-)\left(q(q + 1)^{2g-1} + q^{2g}(q + 1)^{2g-1}\right.
\\ 
&\left.- q^{2g}(q + 1)^{2g-1} - q(q + 1)^{2g-1}\right) +q^{2g+s-1}(q - 1)^{2g}(q + 1)\alpha_+.
\end{align*}

Hence, plugging all these data into formula (\ref{eqn:quot}), we finally get that

\begin{align}\label{eq:pardiag}
	\left[\Char{}(\Sigma_{g,b}, Q)\right] &= (q-1)^{2 g-2} \left(\frac{4 (q-1)^b \alpha _+ \left(q^3-q^{b+2 g+s}\right)}{q^3}\right.\notag\\
	&\left.+(q-1) (q+1) \left(q \left(q^2-1\right)\right)^{b-2}
   q^{2 g+s-1} \left(\left(\frac{q}{q+1}\right)^{-2 g-s+2}\right.\right.\\
   &\left.\left.+(q+1)^{2 g+s-2}+2^{2 g+s-1}-2^s\right)-2 (q-1)^b \alpha _++(q-1)^{b+1} \alpha
   _p\right)\notag\\
   &+ {\cI}_0(\xi_1, \ldots, \xi_s)(q^3-q)^{b-2}\notag
\end{align}



\subsection{The parabolic case with arbitrary punctures}\label{sec:parabolic-arbitrary}

From the results of Sections \ref{sec:parabolic-jordan} and \ref{sec:parabolic-diagonal} we can compute the character variety of a $\SL{2}(k)$-representation variety over the node-surface $\Sigma_{g,b}$ with an arbitrary parabolic structure. Denote by $Q^+_{s}$ the parabolic structure on $\Sigma_{g,b}$ with $s$ punctures with holonomies $[J_+]$, as studied in Section \ref{sec:parabolic-jordan}.

Now, consider an arbitrary parabolic structure $Q_{r_-,r_+}^{t}(\xi_1, \ldots, \xi_s)$ with $t$ punctures with holonomy $[-\Id]$, $r_+$ punctures with holonomy $[J_+]$, $r_-$ punctures with holonomy $[J_-]$ and $s$ punctures with semi-simple holonomies $[D_{\xi_1}], \ldots, [D_{\xi_s}]$. Set $r=(r_+ + r_-)$ and $\sigma = (-1)^{r_- + t}$. Observe that $[J_-]=[-J_+]$ and $(-\Id)^2 = \Id$ so
$$
	\Rep{}(\Sigma_{g,b}, Q_{r_+,r_-}^t(\xi_1, \ldots, \xi_s)) = \left\{\begin{matrix}\Rep{}(\Sigma_{g,b}, Q_{r,0}^0(\xi_1, \ldots, \xi_s)) & \textrm{if } \sigma = 1,\\
	 \Rep{}(\Sigma_{g,b}, Q_{r,0}^1(\xi_1, \ldots, \xi_s))& \,\,\,\,\,\textrm{if } \sigma = -1.
	\end{matrix}\right.	
$$

Then, apart from the cases of Sections \ref{sec:parabolic-jordan} and \ref{sec:parabolic-diagonal}, the remaining cases reduces to the following scenarios.
\begin{itemize}
	\item $r, s > 0$ and $\sigma = 1$. Then the action of $\PGL{2}(k)$ on $\Rep{}(\Sigma_{g,b}, Q)$ is closed and free, since the only matrices that stabilize both a non-diagonalizable and a diagonalizable matrix are the multiples of the identity. Hence, using \cite[Theorem 6.1]{gonzalez2020virtual}, we get that
	\begin{align*}
		\left[\Char{}(\Sigma_{g,b}, Q_{r,0}^0(\xi_1, \ldots, \xi_s))\right] &= \frac{\left[\Rep{}(\Sigma_{g,b}, Q_{r,0}^0(\xi_1, \ldots, \xi_s))\right]}{q^3-q} \\
		&=  q^{2g + s-2}(q - 1)^{2g + r-2}\left(2^{2g+s-1} - 2^s + (q + 1)^{2g + r+s-2} \right)(q^3-q)^{b-1} \\
		& \,\,\,\,\,\,\,+ {\cI}_r(\xi_1, \ldots, \xi_s)(q^3-q)^{b-2},
	\end{align*}
	where the interaction term is
	\begin{align*}
		{\cI}_r(\xi_1, \ldots, \xi_s) =\, & q^{2g + s-1}(q - 1)^{2g + r-1}(\alpha_+ + \alpha_-)\bigg( 2^{2g} + 2^{2g}q  - 2q -2 \\
& \left. + (q + 1)^{2g + r} + (q+1)\left(1 - 2^{2g-1} - \frac{1}{2}(q + 1)^{2g + r-1}\right)\right) \\
& + q^{2g + s-1}(q - 1)^{2g + r}(q+1)\alpha_+
	\end{align*}
	\item $s > 0$ and $\sigma = -1$. Then we can absorbe the $-\Id$ puncture into one of the semi-simple punctures so we have $\Char{}(\Sigma_{g,b}, Q_{r,0}^1(\xi_1, \ldots, \xi_s)) = \Char{}(\Sigma_{g,b}, Q_{r,0}^0(-\xi_1, \ldots, \xi_s))$.
	\item $s = 0$ and $\sigma = -1$. In this case, all the representations of $\Rep{}(\Sigma_{g,b}, Q_{r,0}^1)$ are irreducible. Indeed, if $(A_1, B_1, \ldots, C_1, \ldots, C_{b-1}, P_1, \ldots, P_r) \in \Rep{}(\Sigma_{g,b}, Q_{r,0}^1(\xi_1, \ldots, \xi_s))$ then they satisfy
	$$
		\prod_{i=1}^g [A_i, B_i]\prod_{i=1}^r P_i = -\Id.
	$$
If $v$ is a common eigenvector of these matrices then $\prod_{i=1}^g [A_i, B_i]\prod_{i=1}^r P_i(v) = v$ since the eigenvalues of $[A_i,B_i]$ and $P_i$ are all $1$. This is incompatible with $-\Id(v) = v$ except if $v = 0$. Therefore, $\Rep{}(\Sigma_{g,b}, Q_{r,0}^1) = \Repirred{}(\Sigma_{g,b}, Q_{r,0}^1)$ and, thus, using \cite[Theorem 5.10]{GP-2018}
\begin{align*}
		\left[\Char{}(\Sigma_{g,b}, Q_{r,0}^1)\right] &= \frac{\left[\Rep{}(\Sigma_{g,b}, Q_{r,0}^1)\right]}{q^3-q} = \left(-1\right)^{r + 1}2^{2g - 1}  {\left(q + 1\right)}^{2g + r
+b - 3} {\left(q - 1\right)}^{b-2} q^{2g +b - 3}\\
		& + \, {\left(q - 1\right)}^{2g + r + b - 3} (q+1)^{b-2}q^{2g +b - 3}{{\left( {\left(q + 1\right)}^{2 \,
g + r-2}+2^{2g-1}-1\right)} }.
\end{align*}
\end{itemize}

%
%
%
%
%


\bibliography{bibliography.bib}{}
\bibliographystyle{abbrv}

\end{document}